\documentclass[oneside,english]{amsart}
\usepackage[T1]{fontenc}
\usepackage[latin9]{inputenc}
\usepackage{a4wide}
\usepackage{color}
\usepackage{babel}
\usepackage{wasysym}

\usepackage{todonotes}
\usepackage{refstyle}
\usepackage{units}
\usepackage{textcomp}
\usepackage{amsbsy}
\usepackage{amstext}
\usepackage{amsthm}
\usepackage{amssymb}
\usepackage{amsmath}
\usepackage{subcaption}
\usepackage[unicode=true,pdfusetitle,
bookmarks=true,bookmarksnumbered=false,bookmarksopen=false,
breaklinks=false,pdfborder={0 0 0},pdfborderstyle={},backref=false,colorlinks=true]
{hyperref}
\usepackage{mathptmx}
\makeatletter

\newcommand{\R}{\mathbb{R}}
\newcommand{\PP}{\mathbb{P}}
\newcommand{\Z}{\mathbb{Z}}

\AtBeginDocument{\providecommand\lemref[1]{\ref{lem:#1}}}
\AtBeginDocument{}
\AtBeginDocument{\providecommand\colref[1]{\ref{col:#1}}}
\AtBeginDocument{\providecommand\eqref[1]{\ref{eq:#1}}}
\AtBeginDocument{\providecommand\remref[1]{\ref{rem:#1}}}
\RS@ifundefined{subsecref}
{\newref{subsec}{name = \RSsectxt}}
{}
\RS@ifundefined{thmref}
{\def\RSthmtxt{theorem~}\newref{thm}{name = \RSthmtxt}}
{}
\RS@ifundefined{lemref}
{\def\RSlemtxt{lemma~}\newref{lem}{name = \RSlemtxt}}
{}

\theoremstyle{plain}
\newtheorem{lem}{\protect\lemmaname}
\theoremstyle{definition}
\newtheorem{defn}{\protect\definitionname}
\theoremstyle{plain}
\newtheorem{thm}{\protect\theoremname}
\theoremstyle{plain}
\newtheorem{cor}{\protect\corollaryname}
\theoremstyle{plain}
\newtheorem{prop}{\protect\propositionname}
\theoremstyle{remark}
\newtheorem{rem}{\protect\remarkname}

\def\RSthmtxt{Theorem~}
\def\RSlemtxt{Lemma~}
\def\fu{G}
\AtBeginDocument{}
\RS@ifundefined{propref}{\newref{prop}{name=Proposition~,names=propositions~}}{}

\AtBeginDocument{}
\RS@ifundefined{assuref}{\newref{assu}{name=Assumption~,names=assumptions~}}{}

\AtBeginDocument{\providecommand\remref[1]{\ref{rem:#1}}}
\RS@ifundefined{remref}{\newref{rem}{name=Remark~,names=remarks~}}{}

\AtBeginDocument{\providecommand\colref[1]{\ref{col:#1}}}
\RS@ifundefined{colref}{\newref{col}{name=Corollary~,names=corollaries~}}{}

\makeatother

\providecommand{\definitionname}{Definition}
\providecommand{\lemmaname}{Lemma}
\providecommand{\propositionname}{Proposition}
\providecommand{\remarkname}{Remark}
\providecommand{\corollaryname}{Corollary}
\providecommand{\theoremname}{Theorem}

\begin{document}
	
	\title{The TAZRP speed process}
	
	\author[G.~Amir]{Gideon Amir}
	\address{Gideon Amir\\Bar-Ilan University\\ 5290002\\ Ramat Gan\\ Israel}
	\email{gidi.amir@gmail.com}
	\thanks{G.~Amir was supported by the Israel Science Foundation through grant 575/16 and by the German Israeli Foundation through grant I-1363-304.6/2016.}

	\author[O.~Busani]{Ofer Busani}
	\address{Ofer Busani\\ University of Bristol\\  School of Mathematics\\ Fry Building\\ Woodland Rd.\\Bristol BS8 1UG\\ UK.}
	\email{o.busani@bristol.ac.uk}
	\urladdr{https://people.maths.bris.ac.uk/~di18476/}
	\thanks{O.~Busani was supported by EPSRC's EP/R021449/1 Standard Grant.}
	
	\author[P.~Goncalves]{Patr\'{i}cia Gon\c{c}alves}
\address{Patr\'{i}cia Gon\c{c}alves\\Center for Mathematical Analysis, Geometry and Dynamical Systems\\ Instituto Superior T\'{e}cnico\\Universidade de Lisboa \\ 1049-001 Lisboa \\Portugal}
\email{pgoncalves@tecnico.ulisboa.pt}
\urladdr{{https://patriciamath.wixsite.com/patricia}}
\thanks{P.~Gon\c calves thanks  FCT/Portugal for support through the project
UID/MAT/04459/2013.  This project has received funding from the European Research Council (ERC) under  the European Union's Horizon 2020 research and innovation programme (grant agreement   No 715734)}
	\author[J.~B.~Martin]{James B.\ Martin}
	\address{James B. Martin\\Department of Statistics\\
	University of Oxford\\UK}
	\email{martin@stats.ox.ac.uk}
	\urladdr{{http://www.stats.ox.ac.uk/~martin}}
	
	\begin{abstract}
%	In \cite{amir2011tasep} Amir, Angel and
%	Valk{\'o} introduced the TASEP speed process which allowed them to answer delicate questions about the joint distribution of the speed of several (and infinitely many) second-class particles in the Totally Asymmetric Simple Exclusion Process(TASEP) rarefaction fan. In this paper we introduce the analogue of the TASEP speed process to the Totally Asymmetric Zero Range Process (TAZRP) and use it to obtain new results on the joint distribution of the speed of several second-class particles in the TAZRP with a reservoir.
	
	In \cite{amir2011tasep} Amir, Angel and 	Valk{\'o}
	studied a multi-type version of the totally asymmetric
	simple exclusion process (TASEP) and introduced
	the TASEP speed process, which allowed them to answer delicate
	questions about the joint distribution of the speed of
	several second-class particles in the
	TASEP rarefaction fan. In this paper we introduce
	the analogue of the TASEP speed process for the totally
	asymmetric zero-range process (TAZRP), and use it to obtain
	new results on the joint distribution of the speed
	of several second-class particles in the TAZRP with a reservoir.
	These is a close link from the speed process to
	questions about stationary distributions of multi-type
	versions of the TAZRP; for example we are able to give a
	precise description of the contents of a single site in
	equilibrium for a multi-type TAZRP with continuous labels.

	\end{abstract}
	\maketitle
	
	\section{introduction}

In the totally asymmetric simple exclusion process (TASEP),
each site of $\mathbb{Z}$ contains either a particle or a hole.
If a particle has a hole to its right, they exchange places
at rate $1$. In \cite{FK}, Ferrari and Kipnis considered the TASEP
with Riemannian initial data -- that is, where there exists
an asymptotic density of particles to the left of the origin,
and also a (perhaps different) asymptotic density of particles
to the right of the origin -- and with a second-class particle placed at the origin. The second-class particle interacts with holes as if it were a particle, and with particles as if it were a hole.

As the configuration evolves, the position of the
second-class particle, $X_2(t)$, changes, and a natural
question is whether the limit
\begin{align}\label{scspeed}
U=\lim_{t\rightarrow \infty}t^{-1}X_2(t),
\end{align}
exists, and if so, in what sense. Consider for example
the case where the initial condition has particles at all negative
sites and holes at all positive sites. It was shown in
\cite{FK} that the limit in (\ref{scspeed}) exists
in distribution, and that
\begin{align}
U\sim \text{U}[-1,1],
\end{align}
The hydrodynamics of the TASEP are described by
the inviscid Burgers equation; for these initial conditions,
the equation displays an entire interval of characteristics
emanating from the origin (the so-called ``rarefaction fan"),
and one has the interpretation that the speed of the
second-class particle is distributed uniformly across
the set of characteristics.

%\textcolor{red}{Patricia: In fact they proved not only for TASEP but ASEP and for more general initial conditions. The limit is uniform but the support changes.}

The natural question of whether the convergence in (\ref{scspeed}) can be strengthened to almost sure convergence was resolved in \cite{MG} by Mountford and Guiol, for the rarefaction-fan initial condition, using large deviations for last-passage percolation and microscopic approximation of the Hamilton-Jacobi equation related to the TASEP hydrodynamics. A different proof was given by
Ferrari and Pimentel \cite{FerPim} using a direct
coupling between the path of the second-class particle
and an interface in a two-type last-passage percolation model.

In \cite{FGM2009} Ferrari, Gon{\c c}alves and Martin considered the TASEP process (and partially-asymmetric versions of it) starting from a configuration with two second-class particles
$P$ and $Q$ at positions $0$ and $1$ respectively, with only first class particles to their left and only holes to their right. They showed, for example, that for the TASEP, the probability that $P$ attempts a jump over $Q$ at some time $t>0$ is $\frac{2}{3}$.

In order to answer further questions about the joint distribution of the speed of several second-class particles at the rarefaction fan -- such as, what is the probability that the two second-class particles develop the same speed? -- Amir, Angel and Valk{\'o} \cite{amir2011tasep} introduced the TASEP speed process. In this model one starts from an initial condition in which every site of
$\mathbb{Z}$ contains a particle of a different type, with a
hierarchy determined by their initial position. Each particle
sees itself as a second-class particle viewing all particles to
its left as first-class particles, and all particles to its right as holes. In this way, the particle positioned at any site
$i\in\mathbb{Z}$ develops a speed almost surely, and one obtains the so-called TASEP speed process
\begin{align}
\{U_i\}_{i\in \mathbb{Z}},
\end{align}
a process indexed by $\mathbb{Z}$ which encodes the joint speed of all particles. This process proved to be a rich model encoding much information about the joint behaviour of second-class particles around the rarefaction fan. In the case of two second-class particles in the rarefaction fan, an explicit joint distribution of the speed was obtained, in particular, it was shown that with positive probability ($\frac{1}{6}$) the two particles develop the same speed. In fact, it was shown that with probability $1$,
the set of speeds attained is dense in $[-1,1]$, and
that for any speed $v$ which is attained, there are in fact
infinitely many particles, called a convoy, with speed $v$.

The TASEP speed process was also used in \cite{CH2012} and \cite{C2011} by Coupier and Heinrich to show that in the last-passage percolation model, there are no three geodesics with the same direction. Results from
\cite{amir2011tasep} about the speed process of the TASEP,
and about related questions concerning speeds of particles in
partially asymmetric systems, were recently extended to models with inhomogeneity in space and time by Borodin and Bufetov \cite{borodin2019color}.

A closely related and also widely-studied interacting particle
system is
%Perhaps the second most well studied interacting particle system is
the (constant-rate) Totally Asymmetric Zero-Range Process (TAZRP). In this process each site of $\mathbb{Z}$ can contain any finite number $n$ of particles. Each site is equipped with a Poisson clock with rate 1, upon ringing, if there is a particles at site $x$ it jumps to site $x+1$. Note that for the TASEP the full rarefaction fan is obtained by taking the maximum density ($1$) to the left of the origin and minimum density ($0$) to the right of it. As for the TAZRP the number of particles at each site is unbounded, it seems that the analogue to the full rarefaction fan initial condition for the TASEP is the initial condition where to the left of the origin, the density is infinite and to the right it is zero. This initial condition can be modelled by setting a reservoir for the TAZRP at the left of the origin. The TAZRP with a reservoir is simply the TAZRP on $\{-1,0,1,...\}$ where at site $-1$ there are infinitely many particles. In this model the particles obey the dynamics of TAZRP on $\{0,1,...\}$ while the reservoir itself is equipped with a Poisson clock of rate one, which whenever rings, a particle jumps from the reservoir to site $0$.
In \cite{G2014} it was shown that the TAZRP with a reservoir
has a hydrodynamic limit given by the function
\begin{align}
h(x,t)=
\begin{cases}
\frac{1-\sqrt{\frac{x}{t}}}{\sqrt{\frac{x}{t}}} & \frac{x}{t}\in (0,1)\\
0	& \frac{x}{t}>1.
\end{cases}
\end{align}
In \cite{G2014} Gon\c calves considered  second-class particles  positioned at time $t=0$ at site $0$ for the ZRP with general rate function $g$ and independent Riemannian initial data. In particular, this includes the case with a reservoir at site $-1$ and with all sites to the right of the second-class particle empty.
%\textcolor{red}{Patricia: I proved a more general result than that, I use any ZRP with $g$ and I do not used the coupling for this! We can discuss it, have a look at Theorem 2.2 of my paper.}
Extending a well known coupling between the TASEP and the TAZRP to configurations with second-class particle, it was shown that, in the case of the constant rate TAZRP, the second-class particle has speed $V$ almost surely, and that $V = (\frac{1+U}{2})^2$ where $U\sim \text{U}[-1,1]$. In \cite{BN2017}, Bal{\'a}zs and Nagy obtained the distribution of the speed of a second-class particle at the rarefaction fan for a large set of models including the TASEP and ZRP using a signed measure on the configurations.

In this paper we continue the study of speeds of particles in the TAZRP, and of related questions concerning stationary distributions of multi-type versions of the process.

We consider an initial condition $\eta^*$ in which each site
of $\Z$ has an infinite column of particles (with a bottom particle but no top particle). Every particle
has higher priority than all the particles above it,
and also than all the particles at sites to its right. In this
way, every particle sees itself as a second-class particle
sitting on top of a finite stack of first-class particles at its own site, with an infinite reservoir of first-class particles to its left, and empty space to its right.

We show that every particle develops a speed with probability 1,
leading to an array $U=\{U_{z,i}\}_{z\in \mathbb{Z},i\in \mathbb{N}_0}$ where $U_{z,i}$ is the speed of the particle positioned at column $z$ on top of $i$ particles in $\eta^*$. Furthermore,
the distribution of this ``speed process" $U$ is shown to be
a stationary distribution for a multi-type version of the TAZRP,
whose particles have types in $\R$. Indeed, all translation-invariant stationary distributions can be obtained via appropriate rescalings of the speed process. Although any individual speed
is a continuous random variable, any pair of speeds have positive probability to be equal.

The properties above are analogous to ones known for the TASEP
from \cite{amir2011tasep}. However, in the case of the TAZRP
we can go much further than has been possible for the TASEP
in describing the joint distribution of several speeds. In particular, we give an explicit description of the joint distribution of the speeds of all the particles in a given
column, and hence of the contents of a typical site in
a stationary multi-type TAZRP.

Our approach begins with the coupling between configurations with second-class particles in TASEP and configurations with second-class particles in TAZRP, in particular the connection between the speed of a second-class particle in the TAZRP with the flux of holes seen by a second-class particle in the TASEP. This is combined with the results in \cite{F.M.P.} showing that the second-class particle in the TASEP starting from Riemann initial data has a speed with probability 1, and an expression of the flux of holes seen by a second-class particle as a function of this particle's speed,

To get more precise information about the joint distribution
of speeds, we then develop a new approach involving
fixed points of multi-type queues.
We can think of a site $z$ of the multi-type TAZRP as a
priority queue whose
service process is a Poisson process of rate 1. When a
service occurs, the highest-priority particle present
leaves the queue, moving from $z$ to $z+1$. In a translation-invariant equilibrium, the distribution of the queue's arrival
process (the process of particles moving from $z-1$ to $z$)
is the same as the distribution of the queue's
departure process. Taking as a starting point results
of Martin and Prabhakar \cite{martin2010fixed}, we are
able to build up a detailed description of the
possible distributions of the contents of the queue
for systems with some finite number $n$ of types; by taking
appropriate limits, we can then pass to the full picture
of multi-type equilibria.

The rest of the paper is organized as follows.
In the next section we define the models and give the main results.
In Section \ref{sec:couple} we describe the coupling between the TASEP and the TAZRP, with and without second-class particles. In Section \ref{sec:odd} we prove that distribution of the TAZRP speed process is stationary with respect to the TAZRP dynamics (Theorem \ref{thm:sd}) and start to obtain results on
the distributions of the speeds.
In Section \ref{sec:sq} we study the fixed points
of multi-type priority queues, and prove Theorem
\ref{thm:mq} describing the equilibrium distributions of
a single column in the multi-type TAZRP with a finite
number of types. In Section \ref{ms} we use the results of Section \ref{sec:sq} to prove results about the TAZRP speed process  (Theorem \ref{Poisson pic} and Theorem \ref{thm:tdm}). In Section \ref{sec:ot} we prove a result concerning overtaking between particles which have the same speed (Theorem \ref{thm:ov}).

	\section{main results}\label{sec:mr}
	
	The totaly asymmetric simple exclusion process (TASEP) on $\mathbb{Z}$
	is a Markov process on $\mathcal{Y}=\{0,1\}^{\mathbb{Z}}$ whose generator
	is defined for cylinder functions by $f:\mathcal{Y}\rightarrow \mathbb{R}$
	\begin{equation}
	L^{EP}f\left(\xi\right)=\sum_{x\in\mathbb{Z}}\xi\left(x\right)\left(1-\xi\left(x+1\right)\right)\left(f\left(\xi^{x,x+1}\right)-f\left(\xi\right)\right),\label{eq:TASEP generator}
	\end{equation}
	where
	\[
	\xi^{x,x+1}\left(z\right)=\left\{ \begin{array}{cc}
	\xi\left(x+1\right) & z=x\\
	\xi\left(x\right) & z=x+1\\
	\xi\left(z\right) & \text{otherwise}.
	\end{array}\right.
	\]
	Define the measures $\left\{ \nu_{\alpha}:0\leq\alpha\leq1\right\} $
	as the i.i.d. product measures on $\mathcal{Y}$ s.t $\nu_{\alpha}\left(\xi\left(0\right)=1\right)=\alpha$.
	It is well known that any stationary measure with respect to (\ref{eq:TASEP generator})
	that is also translation invariant is a convex combination of $\left\{ \nu_{\alpha}:0\leq\alpha\leq1\right\} $ (see \cite{liggett2012interacting}).
	Another way to describe the TASEP is through the so-called Harris
	construction. In the Harris construction, we attach to each bond connecting
	two adjacent sites $x$ and $x+1$ a Poisson clock $\mathcal{T}^{\left(x,x+1\right)}$
	of rate one. The dynamics of the process is as follows. At the ring of the clock $\mathcal{T}^{\left(x,x+1\right)}$ at time $t$, if there
	is a particle at site $x$ and no particle at $x+1$ at time $t-$
	then at time $t$ the particle at site $x$ jumps to site $x+1$;
	otherwise, there is no change in the configuration. This construction
	is well defined since on any finite time interval, a.s.\ the graph can be broken into finite subgraphs
	on which the dynamics depends only on its clocks (and not those of
	other subgraphs). \\
	The totally asymmetric zero range process (TAZRP) on $\mathbb{Z}$
	is a Markov process on $\mathcal{\mathcal{X}}=\mathbb{N}_{0}^{\mathbb{Z}}$
	whose generator is given by
	\[
	L^{ZR}f\left(\eta\right)=\sum_{x\in\mathbb{Z}}g\left(\eta\left(x\right)\right)\left(f\left(\eta^{x,x+1}\right)-f\left(\eta\right)\right),
	\]
	where $g:\mathbb{N}_{0}\rightarrow\mathbb{R}_{+}$ satisfies a Lipschitz condition  and
	vanishes at 0, and where
	\[
	\eta^{x,x+1}\left(z\right)=\left\{ \begin{array}{cc}
	\eta\left(x\right)-1 & z=x\\
	\eta\left(x+1\right)+1 & z=x+1\\
	\eta\left(z\right) & \text{otherwise.}
	\end{array}\right.
	\]
	We shall be interested in the case where $g\equiv1$
(i.e.\ $g(x)=1$ for all $x\geq 1$), also known as
	the constant-rate ZRP). As in the TASEP, the TAZRP with $g\equiv1$
	can be built through the Harris construction. On each $z\in\mathbb{Z}$
	a finite number of particles are stacked one on top of the other.
	We attach a Poisson clock $\mathcal{T}^{\left(x,x+1\right)}$ to each
	pair of adjacent sites; upon ringing, if there is at least one particle
	at site $x$ then the bottom particle at $x$ makes a jump to the
	top of the stack at $x+1$, otherwise there is no change in the configuration.
	Alternatively the constant-rate TAZRP can be thought of as a system
	of M/M/1 queues in tandem, one at each site of $\mathbb{Z}$. Using
	the same arguments as before, one can show that the dynamics is well
	defined. The stationary and translation invariant distributions are
	well known for the TAZRP and in the case where $g\equiv1$ are given
	by $\left\{ \mu_{\rho}:0\leq\rho<\infty\right\} $ where $\mu_{\rho}$ is product measure whose marginals are geometric with mean $\rho$, i.e.\
	\begin{align}\label{equilibrium measures for TAZRP}
\mu_{\rho}\left(\eta\left(x\right)=k\right)=\left(\frac{\rho}{1+\rho}\right)^{k}\frac{1}{1+\rho}\qquad k\in\mathbb{N}_{0}.
	\end{align}
	More generally, we can study the \emph{multi-type TAZRP} on
	\begin{align}\label{Zset}
	\mathcal{Z}=\left\{ \eta\in\mathbb{R}^{\mathbb{Z}\times\mathbb{N}_{0}}:\eta\left(z,i\right)\geq\eta\left(z,i+1\right)\right\} .
	\end{align}
	To each particle we assign a ``class'' in $\mathbb{R}$  and now
	the queues become priority queues with infinitely many customers. At each service the highest priority (greatest value)
	particle jumps to the next queue. We will think of the particles at
	each queue as sorted according to their class, with the strongest
	at the bottom.
The value $\eta(z,i)$ represents the class of the
$i$th strongest particle at site $z$.	
	One can get a similar Harris construction using the
	same clocks $\mathcal{T}^{\left(x,x+1\right)}$ as before: when an
	adjacent pair rings the bottom particle at $x$ jumps to $x+1$ and
	positions itself according to its priority. The multi-type TAZRP on
	$\mathcal{Z}$ can be defined through the generator
	\begin{equation}\label{gem}
	Lf\left(\eta\right)=\sum_{x\in\mathbb{Z}}\left(f\left(\sigma_{x}\eta\right)-f\left(\eta\right)\right),
	\end{equation}
	where the operator $\sigma_{x}$ is defined in the following way:
	let $i_{sort^{z}}\left(\alpha\right)=\min\left\{ i:\eta\left(z,i\right)<\alpha\right\} $.
	In other words, $i_{sort^{z}}$ is the lowest index for which $\eta\left(z,\cdot\right)$
	is smaller than $\alpha$. The operator $\sigma_{x}$ is defined
	through
\begin{align}\label{sig}
	\sigma_{x}\eta\left(z,i\right)=\left\{ \begin{array}{cc}
	\eta\left(x,i+1\right) & z=x\\
	\eta\left(x,0\right) & z=x+1,i=i_{sort^{x+1}}\left(\eta\left(x,0\right)\right)\\
	\eta\left(x+1,i-1\right) & z=x+1,i>i_{sort^{x+1}}\left(\eta\left(x,0\right)\right)\\
	\eta\left(z,i\right) & otherwise
	\end{array}\right..
\end{align}
	In words, $\sigma_{x}$ takes the lowest-positioned (and hence of
	highest value) particle in column $x$, $\eta(x,0)$, and puts it at position $i_{sort^{x+1}}\left(\eta(x,0)\right)$
	in column $x+1$ and shifts the position of all particles of value
	 lower than that of $\eta\left(x,0\right)$ upward by one (see Figure \ref{fig:points4}).
	\begin{rem}
		At this point it is not clear why the dynamics in (\ref{gem}) is well defined on the set (\ref{Zset}) as it could be the case that for some $x$
		\begin{align}
			\eta(x,0)<\inf_i \eta(x+1,i).
		\end{align}
		Nevertheless, we shall point out where needed, why on the set of configurations in $\mathcal{Z}$ the dynamics is well defined.
	\end{rem}
	We shall also need the operator $\sigma_{x}^{*}$ on $\mathcal{Z}$
	, which takes $\eta\left(x+1,0\right)$ and puts it in the
	correct position
	in column $x$. More precisely, we define
\begin{align}\label{sig2}
		\sigma_{x}^{*}\eta\left(z,i\right)=\left\{ \begin{array}{cc}
	\eta\left(x+1,i+1\right) & z=x+1\\
	\eta\left(x+1,0\right) & z=x,i=i_{sort^{x}}\left(\eta\left(x+1,0\right)\right)\\
	\eta\left(x,i-1\right) & z=x,i>i_{sort^{x}}\left(\eta\left(x+1,0\right)\right)\\
	\eta\left(z,i\right) & otherwise
	\end{array}\right..
\end{align}
%%%% figure for the dynamics
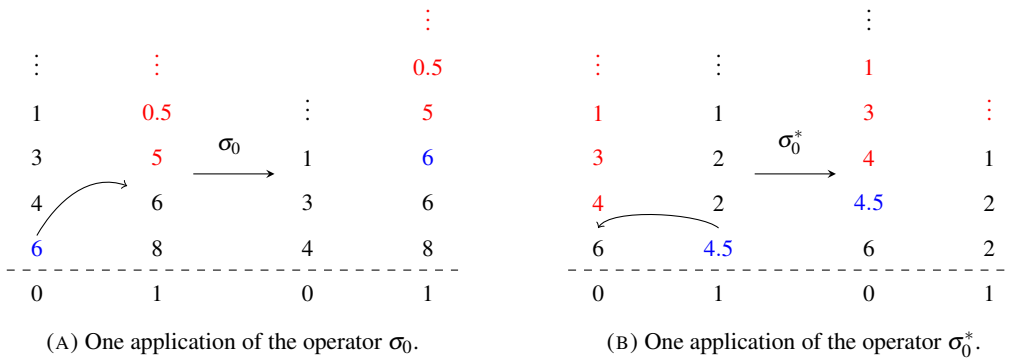
\begin{figure}[ht!]
	\centering%
	\begin{subfigure}[t]{.5\textwidth}
		\centering
		\def\y{2.2}
		\def\z{4}
		\begin{tikzpicture}[scale=0.4, every node/.style={transform shape}]
		\foreach \x in {0,...,1}
		{
			\node [scale=\y][above] at (\x*\z,-1.5) {$\x$};
		};
		%next column
		\node [scale=\y][above][blue] at (0,0) {$6$};
		\node [scale=\y][above][] at (0,1.5) {$4$};
		\node [scale=\y][above][] at (0,3) {$3$};
		\node [scale=\y][above][] at (0,4.5) {$1$};
		\node [scale=\y][above][] at (0,6) {$\vdots$};
		%next column
		\node [scale=\y][above][] at (\z,0) {$8$};
		\node [scale=\y][above][] at (\z,1.5) {$6$};
		\node [scale=\y][above][red] at (\z,3) {$5$};
		\node [scale=\y][above][red] at (\z,4.5) {$0.5$};
		\node [scale=\y][above][red] at (\z,6) {$\vdots$};
		\draw[->] (0,0.95)  to [out=60,in=160, looseness=1] (\z-1,2.6);
		%second step
		\def\w{9}
		\foreach \x in {0,...,1}
		{
			\node [scale=\y][above] at (\x*\z+\w,-1.5) {$\x$};
		};
		%next column
		%next column
		\node [scale=\y][above][] at (\w,0) {$4$};
		\node [scale=\y][above][] at (\w,1.5) {$3$};
		\node [scale=\y][above][] at (\w,3) {$1$};
		\node [scale=\y][above][] at (\w,4.5) {$\vdots$};
		%next column
		\node [scale=\y][above][] at (\z+\w,0) {$8$};
		\node [scale=\y][above][] at (\z+\w,1.5) {$6$};
		\node [scale=\y][above][][blue] at (\z+\w,3) {$6$};
		\node [scale=\y][above][red] at (\z+\w,4.5) {$5$};
		\node [scale=\y][above][red] at (\z+\w,6) {$0.5$};
		\node [scale=\y][above][red] at (\z+\w,7.5) {$\vdots$};
		\draw [->,>=stealth] (1.3*\z,3) -- (1.3*\z+0.3*\w,3);
		\node [scale=\y][above] at (1.6*\z,3.4) {$\sigma_0$};
		\draw [dashed]  (-1,-0.2) -- (14,-0.2);
		\end{tikzpicture}
		\caption{\small One application of the operator $\sigma_{0}$.}
	\end{subfigure}%
	\begin{subfigure}[t]{.5\textwidth}
		\centering
		\def\y{2.2}
		\def\z{4}
		\begin{tikzpicture}[scale=0.4, every node/.style={transform shape}]
		\foreach \x in {0,...,1}
		{
			\node [scale=\y][above] at (\x*\z,-1.5) {$\x$};
		};
		%next column
		\node [scale=\y][above][] at (0,0) {$6$};
		\node [scale=\y][above][red] at (0,1.5) {$4$};
		\node [scale=\y][above][red] at (0,3) {$3$};
		\node [scale=\y][above][red] at (0,4.5) {$1$};
		\node [scale=\y][above][red] at (0,6) {$\vdots$};
		%next column
		\node [scale=\y][above][blue] at (\z,0) {$4.5$};
		\node [scale=\y][above][] at (\z,1.5) {$2$};
		\node [scale=\y][above][] at (\z,3) {$2$};
		\node [scale=\y][above][] at (\z,4.5) {$1$};
		\node [scale=\y][above][] at (\z,6) {$\vdots$};
		\draw[->] (\z,1.2)  to [out=140,in=60, looseness=0.5] (0,1.2);
		%second step
		\def\w{9}
		\foreach \x in {0,...,1}
		{
			\node [scale=\y][above] at (\x*\z+\w,-1.5) {$\x$};
		};
		%next column
		%next column
		\node [scale=\y][above][] at (\w,0) {$6$};
		\node [scale=\y][above][blue] at (\w,1.5) {$4.5$};
		\node [scale=\y][above][red] at (\w,3) {$4$};
		\node [scale=\y][above][red] at (\w,4.5) {$3$};
		\node [scale=\y][above][red] at (\w,6) {$1$};
		\node [scale=\y][above][] at (\w,7.5) {$\vdots$};
		%next column
		\node [scale=\y][above][] at (\z+\w,0) {$2$};
		\node [scale=\y][above][] at (\z+\w,1.5) {$2$};
		\node [scale=\y][above][][] at (\z+\w,3) {$1$};
		\node [scale=\y][above][red] at (\z+\w,4.5) {$\vdots$};
		\draw [->,>=stealth] (1.3*\z,3) -- (1.3*\z+0.3*\w,3);
		\node [scale=\y][above] at (1.6*\z,3.4) {$\sigma^*_0$};
		\draw [dashed]  (-1,-0.2) -- (14,-0.2);
		\end{tikzpicture}
		\caption{\small One application of the operator $\sigma^*_{0}$.}
	\end{subfigure}%
	\caption{\small The two operators $\sigma$ and $\sigma^*$ acting on columns $0$ and $1$.}
	\label{fig:points4}
\end{figure}
%\section{the TAZRP speed process}
We would like to consider
a process analogous to the TASEP speed process introduced in \cite{amir2011tasep}. In \cite{amir2011tasep}, an ergodic process $\{U_i\}_{i\in \mathbb{Z}}$ was constructed where $U_0 \sim \text{U}[-1,1]$. The marginal $U_i$, represents the speed of a second-class particle positioned between infinite first class particles to its left and infinite holes to its right, under the TASEP dynamics. The coupling between the different marginals is obtained by starting from an initial condition where there is a hierarchy between the particles. In this initial condition, each particle is stronger than its neighbour to its right. On the ring of the bell at the edge connecting sites $x$ and $x+1$, the particle at site $x$ jumps to site $x+1$ if and only if the particle at site $x+1$ is weaker (i.e. higher class) than the particle at $x$.

In the ZRP the number of particles at each site is not bounded.
We consider starting the dynamics from a configuration where each site has an infinite number of particles.
%\begin{align}\label{par}
%	p_{z,n}\text{  -  the particle sitting at column $z$ on top of %$n$ particles.}
%\end{align}

We denote by $p_{z,n}$ the particle sitting at site $z$
with $n$ other particles below it.
Here too, we impose a full hierarchy (order relation) on the initial particles according to the lexicographical order: $p_{i,j}$ is stronger than $p_{k,l}$ (denoted $p_{i,j}>p_{k,l}$) if $i<k$, or if $i=k$ and $j<l$. (That is, each particle is stronger than those at sites to its right, or at the same site and directly above it).

Consider a specific particle in our initial configuration. If we only care about the dynamics of that particle, then we do not care about the hierarchy between the other particles. Thinking of that particle as second-class, we may consider  all particles underneath it or at sites to its left as first-class particles, and all the particles above it or at sites to its right as holes. We will show in  Section \ref{sec:couple} that, under the multi-type TAZRP dynamics, this particle will develop a speed with probability $1$. We record this speed in $U_{i,j}$, the $(i,j)$'th element of the array $\{U_{i,j}\}_{i,j\in \mathbb{Z}\times \mathbb{N}_0}$, the TAZRP speed process.

Another way to visualize the configuration of particles
denoted above by $p_{z,n}$
is by considering an array of numbers $\eta ^* \in \mathcal{Z}$ with
\begin{equation}\label{initial condition}
	\eta^*(z,i) < \eta^*(w,j) \text{ if and only if } (w=z \text{ and } i>j) \text{ or } (w<z)
\end{equation}
where $\mathcal{Z}$ is the set defined in (\ref{Zset}) (see Figure \ref{fig:points3}). Here, particle $p_{z,i}$ is identified with the number $\eta^*(z,i)$ and the index $(z,i)$ throughout the dynamics. The number $\eta^*(z,i)$  plays the role of the class, or type of the particle which  determines its interaction with other particles in the configuration in time. Note that stronger particles correspond to higher values, as opposed to the set-up in \cite{amir2011tasep}. Between each pair of neighbouring columns in the array we assign a Poisson clock, where upon ringing the largest number in the left column (sitting at the bottom of the column) makes a jump to the right column and positions itself on top of all the numbers that are strictly larger than itself. We shall go from the picture of the array of numbers to the array of particles often. We also use the words class and type interchangeably, so that
\begin{align}
	\text{$p_{z,i}$ has higher class than $p_{w,j}$ \,\,  $\Leftrightarrow$ \,\,  $\eta^*(z,i) > \eta^*(w,j)$}.
\end{align}
%%%%%% figure showing \eta^* and the dynamics
\begin{figure}[ht!]
		\centering%
		\def\y{2.8}
		\def\z{5}
		\begin{tikzpicture}[scale=0.4, every node/.style={transform shape}]
		\foreach \x in {-1,...,2}
		{
			\node [scale=\y][above] at (\x*\z,-1.5) {$\x$};
		};
		\node [scale=\y][above] at (-\z-2,4) {$\dots$};
		\node [scale=\y][above] at (-\z,0) {$p_{-1,0}$};
		\node [scale=\y][above] at (-\z,2) {$p_{-1,1}$};
		\node [scale=\y][above] at (-\z,4) {$\vdots$};
		\node [scale=\y][above] at (-\z,6) {$p_{-1,i}$};
		\node [scale=\y][above] at (-\z,8) {$\vdots$};
		%next column
		\node [scale=\y][above] at (0,0) {$p_{0,0}$};
		\node [scale=\y][above] at (0,2) {$p_{0,1}$};
		\node [scale=\y][above] at (0,4) {$\vdots$};
		\node [scale=\y][above] at (0,6) {$p_{0,i}$};
		\node [scale=\y][above] at (0,8) {$\vdots$};
		%next column
		\node [scale=\y][above] at (\z,0) {$p_{1,0}$};
		\node [scale=\y][above] at (\z,2) {$p_{1,1}$};
		\node [scale=\y][above] at (\z,4) {$\vdots$};
		\node [scale=\y][above] at (\z,6) {$p_{1,i}$};
		\node [scale=\y][above] at (\z,8) {$\vdots$};
		%next column
		\node [scale=\y][above] at (\z*2,0) {$p_{2,0}$};
		\node [scale=\y][above] at (\z*2,2) {$p_{2,1}$};
		\node [scale=\y][above] at (\z*2,4) {$\vdots$};
		\node [scale=\y][above] at (\z*2,6) {$p_{2,i}$};
		\node [scale=\y][above] at (\z*2,8) {$\vdots$};
		\node [scale=\y][above] at (\z*2+2,4) {$\dots$};
		\end{tikzpicture}
	\caption{\small The initial configuration $\eta^*$}
	\label{fig:points3}
\end{figure}
%  We divide the particles into classes, that is, each
%particle is assigned a real number that determines the class of the
%particle. At any column $z$, particles are ordered according to their
%classes, that is, particles of the lowest class are positioned at
%the bottom (the order within the class is unimportant), above them
%are the particle of next to lowest, and so on. At the ring of a Poisson
%clock assigned to site $z$ the particle at the bottom jumps to site
%$z+1$ and positions himslef exactly below the particles that are
%of the lowest class that is higher then his. It is not hard to see
%that the process described above can be realized through the process
%defined in the previous section through the generator (\ref{gem}).
%Denote the particle that sits on the $z$ site at height $n$ by $p_{z,i}$
%(there exist $n-1$ particles under it. We imagine the particles sitting
%at site $z$ stacked one on top of the other). We shall be interested
%in the initial configuration where the particle $p_{z,i}$ has a lower
%class then $p_{j,m}$ if $i<j$ or $i=j$ and $n<m$. In the context
%of the multi-type TAZRP, we can set the initial condition to be, for
%instance,
%\[
%\eta^{*}\left(z,i\right)=-z-\sum_{j=1}^{i}3^{-j}.
%\]
Let $X_{z,i}\left(t\right)$ denote the position of the particle $p_{z,i}$
at time $t$, that is, the site (column) $X_{z,i}\left(t\right)\in\mathbb{Z}$
where the particle $p_{z,i}$ can be found at time $t$ under the
dynamics of (\ref{gem}) and the initial condition $\eta^{*}$, i.e. the multi-type TAZRP. Let $p_{z,i}$
and $p'_{w,i}$ be two particles in the configurations $\eta$ and
$\eta'$ respectively. We say the particle $p_{z,i}$ sees the same
environment as particle $p'_{w,i}$ if for every $l\in\mathbb{N}_{0}$
and $k\in\mathbb{Z}$
\begin{align*}
p_{z+k,l} & \leq p_{z,i}\Longleftrightarrow p'_{w+k,l}\leq p'_{w,i},\\
p_{z+k,l} & \geq p_{z,i}\Longleftrightarrow p'_{w+k,l}\geq p'_{w,i}.
\end{align*}
 In other words, two particles see the same environment if the relative order between them and other particles in the configuration is preserved relative to their position.\\
We say that $p_{z,i}$ has a speed
if $\lim_{t\rightarrow\infty}t^{-1}X_{z,i}\left(t\right)$ exists
and call the limit the speed of $p_{z,i}$. We have the following result, a proof of which we give in  Section \ref{sec:couple}.
\begin{lem}\label{Fpro}
For every $z \in\mathbb{Z}$ and $i\in\mathbb{N}$ the particle $p_{z,i}$ has a strictly positive speed with probability one.  That is, the following limit exists and is strictly positive a.s.
\begin{equation}
U_{z,i}=\lim_{t\rightarrow\infty}t^{-1}X_{z,i}\left(t\right)>0.\label{eq:speed_of_SCP}
\end{equation}
\end{lem}

We are now in a position to define the TAZRP speed process.
\begin{defn}\label{def}
The \emph{TAZRP speed process} $U=\{U_{z,i}\}_{\mathbb{Z}\times \mathbb{N}_0}$ is given by
\[
U_{z,i}=\lim_{t\rightarrow\infty}t^{-1}X_{z,i}\left(t\right) \qquad \text{for $z\in \mathbb{Z},i\in \mathbb{N}_0$},
\]
where the limit is with probability one. We define $\mu$ to be the distribution of the process $U$ on $\mathcal{Z}$.
\end{defn}

We also show the following property of the TAZRP speed process, which together with the previous Lemma show that $U\in \mathcal{Z}$ and that on the support of $\mu$ in the set $\mathcal{Z}$ the dynamics in (\ref{gem}) is well defined.
\begin{lem}\label{lem:Esum=1}
$\sum_{i=0}^{\infty} \mathbb{E}[U_{z,i}]=1$ and $\PP(\inf_{0 \leq i} U_{z,i}=0)=1$. 
\end{lem}

Note that although, clearly, the measure $\mu$ is translation invariant,
it may not be reflection invariant. Let $\pi$ denote the reflection
operator on $\mathcal{Z}$ defined by $\pi\eta\left(x\right)=\eta\left(-x\right)$.
Then $\pi$ operates on measures on $\mathcal{Z}$ in the usual way,
and we define $\mu^{\pi}=\pi\mu$. We also denote by $\mu^\pi_0$ the distribution of the 0'th column of $\mu^\pi$ (and by stationarity the distribution of any column).

%\begin{thm}\label{thm:sd}
%The distribution $\mu^{\pi}$
%is the unique stationary and ergodic measure of the multi-type TAZRP dynamics whose marginal column distribution is $\mu^\pi_0$.
%\end{thm}

	Let $\fu:\mathbb{R}\rightarrow\mathbb{R}$ be a non-decreasing
function. For $\eta\in\mathcal{Z}$, we write $\fu\left(\eta\right)$
for the configuration $\fu\left(\eta\right)_{z,i}=\fu\left(\eta_{z,i}\right)$.
Note that $\fu\left(\eta\right)\in\mathcal{Z}$. An easy yet important observation is that the dynamics of the multi-type TAZRP
(and likewise the TASEP) are conserved under a monotone relabelling
of the types. (See Lemma \ref{lem:preservence_of_monotonicity}
and Corollary \ref{col:SIM} in Section \ref{sec:odd} below.)

We can now state our first main result.

\begin{thm}\label{thm:sd}
The distribution $\mu^{\pi}$
is an ergodic stationary distribution of the multi-type TAZRP.
Any other translation-invariant ergodic stationary distribution
is the distribution of $\fu(\eta)$ where $\eta\sim \mu^\pi$,
for some non-decreasing function $\fu$ from $\R$ to $\R$.
\end{thm}

Our next result is about the stationary measures for the $n$-type TAZRP. In the $n$-type TAZRP there are $n$ different classes of particle which may be present at any site.
The first-class particles have the highest priority,
followed by the second-class, and so on.
We may imagine the particles at a site (or column) ordered
according to their type, with the highest-priority particles
at the bottom. When the clock rings at site $x$, the particle of the highest priority jumps to site $x+1$ and positions itself according to its class in column $x+1$. The $n$-type TAZRP can be obtained by restricting the multi-type TAZRP to a subset of the set $\mathcal{Z}$ in (\ref{Zset}). Let
$\mathcal{Z}_{n}=\mathcal{Z}\cap\mathcal{R}_{n}^{\mathbb{Z}\times\mathbb{N}_{0}}$,
where $\mathcal{R}_{n}=\left\{ -1,..,-n, -n-1\right\} $. Let $\eta\left(t\right)$
be the multi-type TAZRP on $\mathcal{Z}$. Then the set $\mathcal{Z}_{n}$
is closed under the dynamics of $\eta\left(t\right)$ , that is, if
$\eta\left(0\right)\in\mathcal{Z}_{n}$ then $\eta\left(t\right)\in\mathcal{Z}_{n}$
for all $t>0$ . We define the $n$-type TAZRP to be the multi-type
TAZRP restricted to $\mathcal{Z}_{n}$.

The interpretation is that for $i=1,2,\dots, n$,
a particle of type $i$ has a label of type $-i\in\mathcal{R}_n$.
If the total number of particles of types $1$ up to $n$ at a site $x$ in a configuration $\eta\in\mathcal{Z}_n$ is $k$, then
$\eta_{x,i}\geq-n$ for $i\leq k-1$ and $\eta_{x,i}=-n-1$ for $i\geq k$.
We interpret the label $-n-1$ as describing
a ``hole" or ``absence" of a particle in the $n$-type TAZRP.
The choice of $\mathcal{R}_{n}$ is not crucial -- one could take any ordered set of size $n+1$ -- but using the set $\mathcal{R}_{n}$, one can  read off the class of the particles by removing the minus sign from the particle label. The TAZRP is the $n$-type TAZRP for $n=1$.

Let $\alpha,p\in(0,1)$. We say a random variable $X$ has
geometric distribution with parameter $\alpha$, denoted
$X\sim\text{Geom}(\alpha)$,
if $\PP(X=k)=(1-\alpha)\alpha^{k-1}$ for $k\geq 1$,
and that $X$ has a Bernoulli-geometric
distribution, denoted $X\sim\text{Ber}(p)\text{Geom}(\alpha)$, if
\[
\mathbb{P}\left(X=k\right)=\begin{cases}
\left(1-p\right) & k=0\\
p\left(1-\alpha\right)\alpha^{k-1} & k\geq 1.
\end{cases}
\]
For $1 \leq i \leq n$ let us denote by $Q_i$ the number of particles of class $i$ in the column $0$ of configurations in $\mathcal{Z}_n$.
\begin{thm}\label{thm:mq}
For any translation-invariant ergodic stationary distribution of the
$n$-type TAZRP,
with non-zero and finite density of particles of types $1,2,\dots,n$, there are $\lambda_1, \dots,\lambda_n>0$
with $\sum_{i=1}^n \lambda_i<1$
such that
the random variables $Q_i$ are independent, and
	\begin{align}\label{Qd}
	Q_{i}\sim \mathrm{Ber}\left(\frac{\lambda_{i}}{1-\left(\lambda_{1}+...+\lambda_{i-1}\right)}\right)\mathrm{Geom}\left(\lambda_{1}+...+\lambda_{i}\right).
	\end{align}
\end{thm}	
	
	\begin{rem}
		\label{rem:in_thm_Q_i}
Note that as one should expect from the well-known
geometric i.i.d.\ distribution of the TAZRP,
or basic results on stationary distributions on
stationary distributions of $M/M/1$ queues in series,
the distribution of $Q_1$ is geometric,
with parameter $\lambda_1$.
In general, the sum a geometric and an independent
Bernoulli-geometric is not geometric, but
for particular values of the parameters such
relations do hold, and in this case one obtains,
also as expected, that $Q_1+\dots+Q_i$ is geometric for each $1\leq i\leq n$,
with parameter $\lambda_1+\dots+\lambda_i$.		
We may also intepret $\lambda_i$ as the intensity
at which particles of type $i$ move from site $0$
to site $1$. Also $\lambda_i$ is the probability
that the highest-priority particle at site $0$ is of type $i$.
%		
%		
%		
%		
%	\begin{enumerate}
%		\item Here by $X\sim \text{Geom}(p)$ we mean that $\mathbb{P}(X=x)=(1-p)p^{x-1}$.
%		\item The distribution of $Q_{1}$ is geometric ($\mathbb{P}(Q_1=x)=(1-p)p^{x}$) as one would expect from
%		the well known geometric i.i.d. distribution of the TAZRP or basic
%		results on stationarity of queues.
%		\item There is a calculus of these Bernoulli-geometric distributions, under
%		which adding a geometric to an independent Bernoulli-geometric gives
%		another geometric distribution, if the parameters are appropriate.
%		In that way it works out that $Q_{1}+...+Q_{i}$ is geometric with
%		parameter $\lambda_{1}+...+\lambda_{i}$.
%	\end{enumerate}
\end{rem}

%\begin{thm}\label{thm:mq}
%	For every $0<\lambda_1,...,\lambda_n$ such that
%	\begin{align}
%		\sum_{i=1}^n\lambda_i<1
%	\end{align}
%	 there exists a unique ergodic and stationary distribution $\mu^{\lambda_1,...,\lambda_n}$ of the $n$-type TAZRP under which the random variables $Q_i$ are independent and
%	\begin{align}\label{Qd}
%	Q_{i}\sim \mathrm{Ber}\left(\frac{\lambda_{i}}{1-\left(\lambda_{1}+...+\lambda_{i-1}\right)}\right)\mathrm{Geom}\left(\lambda_{1}+...+\lambda_{i}\right).
%	\end{align}
%	\begin{rem} $\quad$ \\
%		\label{rem:in_thm_Q_i}
%	\begin{enumerate}
%		\item Here by $X\sim \text{Geom}(p)$ we mean that $\mathbb{P}(X=x)=(1-p)p^{x-1}$.
%		\item The distribution of $Q_{1}$ is geometric ($\mathbb{P}(Q_1=x)=(1-p)p^{x}$) as one would expect from
%		the well known geometric i.i.d. distribution of the TAZRP or basic
%		results on stationarity of queues.
%		\item There is a calculus of these Bernoulli-geometric distributions, under
%		which adding a geometric to an independent Bernoulli-geometric gives
%		another geometric distribution, if the parameters are appropriate.
%		In that way it works out that $Q_{1}+...+Q_{i}$ is geometric with
%		parameter $\lambda_{1}+...+\lambda_{i}$.
%	\end{enumerate}
%\end{rem}
%\end{thm}

Now that we have stated the result that the distribution of the speed process $U$ is a stationary measure for the multi-type TAZRP dynamics (Theorem \ref{thm:sd}), we turn to investigating this measure. As each of the second-class particles has speed, the column of speeds $\{U_{0,i}\}_{i\in \mathbb{N}_0}$ can be thought of as a marked point process where the points are the set of speeds in $[0,1]$ attained by the particles at column zero and the mark associated with the point $v\in[0,1]$ is the number of particles attaining a specific speed. The following result characterizes the distribution of a column of the speed process.
%\begin{thm}\label{thm:ocd}
%	Let $U$ be a translation
%	invariant stationary distirbution for the  multi-type TAZRP with
%	types in $[0,1]$ such that the rate of arrival of particles with
%	types in $[x,1]$ is $1-x$ for all $x$ in $[0,1]$. Then the content
%	of column $0$ (under $U$) can be described by a Poisson
%	point process on $[0,1]$ with intensity $1/x$ with each point of
%	the Poisson process contributing $\text{Geom}((1-x))$ particles of
%	type $x$ to the queue.
%\end{thm}
\begin{thm}\label{Poisson pic}
	Let $U$ be the speed process. The distribution of $\left\{ U_{0,i}\right\} _{i\in\mathbb{N}_0}$
	is a marked Poisson process on $[0,1]$, with intensity $\frac{1}{2x}$ and mark distribution $Geom\left(1-\sqrt{x}\right)$. In particular, for a fixed $j>0$, the sequence of speeds $\left\{ U_{0,i}\right\} _{i=j+1}^{\infty}$
	conditioned on $U_{0,j}$ is independent of $\left\{ U_{0,i}\right\} _{i=0}^{j-1}$.
\end{thm}
Theorem \ref{Poisson pic} shows that the set of values of $U_{0,\cdot}$ accumulates at $0$. We also see that conditioning on some particle attaining the speed $v$, the probability of finding another particle with speed $v$ is positive. Moreover, it gives a Markovian property for the column of speeds.

Note that a related marked-Poisson-process structure
was recently found by Fan and Seppalainen \cite{FanSeppalainen}
in the description of joint distributions of
Busemann functions for the last-passage percolation model
(see for example their Theorem 3.4).

Our next result states the joint distribution of two second-class particles starting at column $0$.
\begin{thm}\label{thm:tdm}
	Let $U$ be the TAZRP speed process
	 and let $f(x)=1-\sqrt{x}$. Then for $i<j$ and $x_1>x_2$
	\[
	\mathbb{P}\left(x_{1}\geq U_{0,i}, x_{2}\geq U_{0,j}\right)=1-f(x_1)^{i+1}-f(x_2)^{j+1}\left(1-\left(\frac{f(x_1)}{f(x_2)}\right)^{i+1}\right),
	\]
	and
	\[
	\mathbb{P}\left(U_{0,i}=U_{0,j}\in dx\right)=(i+1)\frac{f(x)^j}{2\sqrt{x}}dx.
	\]
\end{thm}
Theorem \ref{thm:tdm} again says that there is a positive probability for two second-class particles at a column to have the same speed; we also see that conditional on having the same speed, the distribution of the speed has a density.

In general, obtaining results on the joint distribution of two columns is hard. We have the following result in this direction.
\begin{prop}\label{tc}
	Let $U$ be the speed process of the TAZRP, and let $f(x)=1-\sqrt{x}$ and $j,k\in\mathbb{N}$.
	Then,
	\begin{align}
	&\mathbb{P}\left(U_{0,0}>x_1,U_{-1,j-1}>x_{1}>U_{-1,j}>...>U_{-1,j+k-1}>x_{2}\right)\\
	&=\left(f(x_2)-f(x_1)\right)f(x_1)^{j}f(x_2)^{k}.\nonumber
	\end{align}
\end{prop}
Take two particles, $p_{0,j}$ and $p_{i,k}$ where $0<i$. Both particles develop speed $v$ and $u$ respectively. On the event that $v=u$, what is the probability that $p_{0,j}$ overtakes $p_{i,k}$? that is, what is the probability that $X_{0,j}(t)>X_{i,k}(t)$ for some $t>0$? Our next result shows that overtaking occurs with probability 1.
\begin{thm}\label{thm:ov}
	Let $U$ be the TAZRP speed process. Suppose $i>0$ and condition on $U_{0,j}\geq U_{i,k}$. Then with probability 1, $p_{0,j}$
	overtakes $p_{i,k}$.
\end{thm}

\section{the coupling between  TASEP and  TAZRP}\label{sec:couple}

\subsection{The basic coupling}\label{ssec:CTZ}
We begin by describing a coupling between the exclusion and zero-range process on $\mathbb{Z}$. There are, in fact, two natural ways to define such a coupling,
particle-hole and particle-particle, and both work for ASEP-AZRP as well.
In the particle-hole coupling, each particle in the TASEP configuration will correspond to a column in the TAZRP, and consecutive holes between particles in the TASEP correspond to particles sitting in the same column (the column that corresponds to the first particle to their left). The advantage of this coupling is that clocks on the particles in the TASEP correspond naturally to clocks on the sites (columns) in the TAZRP, though the direction of movement is reversed. This coupling was originally introduced by Kipnis in \cite{K.}, where he used it to relate several observables between TASEP and TAZRP, e.g. the position of a  tagged particle at time $t$ in the TASEP with the current through a bond up to time $t$ in the TAZRP.\\
We will be more interested in the particle-particle coupling as it can be generalized to deal with second-class particles.
In the particle-particle coupling, each hole in the TASEP configuration corresponds to a column in the TAZRP, and the particles between consecutive holes become particles sitting in the column corresponding to the first hole to their right. As this is the coupling we plan to use, we describe it more rigorously.
Let $\xi_t$ be a TASEP configuration. Denote by $\{y_i(0)\}_{i\in\mathbb{Z}}$  the positions of all the holes at time $0$, ordered so that $...<y_{-1}(0)<y_{0}(0)<y_{1}(0)<..$ , with $y_0(0)$ denoting the first hole in position $> 0$ at time $0$. Let $y_i(t)$ denote the position of the $i$'th hole at time $t$. We construct a configuration $\eta_t$ from
$\xi_t$ by setting $\eta_t(i)=y_{i+1}(t)-y_i(t)-1$.
It is not hard to check that under this coupling $\eta_t$ follows TAZRP dynamics, and that, in fact, the clocks on the columns correspond to the clocks of the particles in the TASEP that may indeed jump. We denote by $\Phi:\mathcal{Y}\rightarrow \mathcal{X}$ (Figure \ref{fig:tzm}) the mapping between TASEP and TAZRP configurations described above.
%Below we see two initial configurations for TASEP (left hand side) and the corresponding configuration in TAZRP (right hand side). The first class particles will be represented by $\CIRCLE$; holes will be represented by $\Circle$. The particle at the origin is underlined.
%
%
%\vspace{0.2cm}
%
%
%
%
%\hspace{3mm}
%$ \CIRCLE\hspace {3.5mm} \CIRCLE $
%
%
%$\ldots \CIRCLE \CIRCLE{\underline{\CIRCLE}}\CIRCLE \Circle \Circle \CIRCLE \Circle \ldots$ \hspace{5mm}
%$\ldots \Circle \CIRCLE \CIRCLE \Circle  \CIRCLE \Circle \underline{\CIRCLE} \CIRCLE \Circle \CIRCLE \Circle \Circle \Circle\CIRCLE \Circle \Circle \ldots$
\begin{figure}[t!]
	\centering%
	\begin{tikzpicture}[scale=0.4, every node/.style={transform shape}]
	\def\h{5}
	\def\y{1.5}
	\def\t{-1.5}
	%	\foreach \x in {-5,...,2}
	%{
	%	\node [scale=\y][above] at (\x,-1.5) {$\x$};
	%};

	%second step
	\node [scale=2][above] at (0,-0.7) {$0$};
	\node [scale=2.5][above] at (-4,0) {$\CIRCLE$};
	\node [scale=2.5][above] at (-3,0) {$\Circle$};
	\node [scale=2.5][above] at (-2,0) {$\CIRCLE$};
	\node [scale=2.5][above] at (-1,0) {$\CIRCLE$};
	\node [scale=2.5][above] at (0,0) {$\Circle$};
	\node [scale=2.5][above] at (1,0) {$\CIRCLE$};
	\node [scale=2.5][above] at (2,0) {$\CIRCLE$};
	\node [scale=2.5][above] at (3,0) {$\Circle$};
	\node [scale=2.5][above] at (4,0) {$\Circle$};

	%third step
	\node [scale=2][above] at (10,-0.7) {$0$};	
	\node [scale=2.5][above] at (8,0) {$\CIRCLE$};
	
	\node [scale=2.5][above] at (9,0) {$\CIRCLE$};
	\node [scale=2.5][above] at (9,1) {$\CIRCLE$};
	
	\node [scale=2.5][above] at (10,0) {$\CIRCLE$};
	\node [scale=2.5][above] at (10,1) {$\CIRCLE$};
	%\node [scale=2.5][above] at (10,-3) {$\circledast$};
	\node [scale=2.5][above] at (11,0) {$\Circle$};
	\draw [->,>=stealth] (5.5,1) -- (6.5,1);
	\node [scale=2.5][above] at (5.75,1.25) {$\Phi$};
	\end{tikzpicture}
	\caption{\small The mapping $\Phi$ between the TASEP and the TAZRP.}
	\label{fig:tzm}
\end{figure}
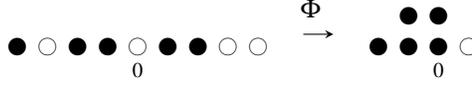
\vspace{0.15cm}
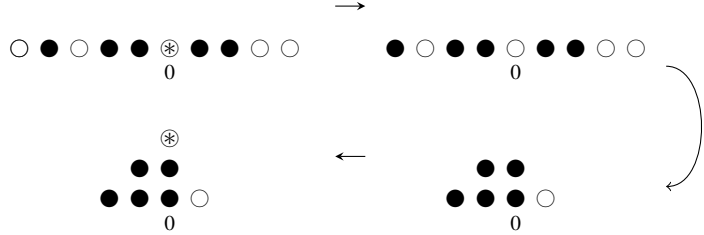
\begin{figure}[ht!]
	\centering%
	\begin{tikzpicture}[scale=0.4, every node/.style={transform shape}]
	\def\h{5}
	\def\y{1.5}
	\def\t{-1.5}
	%	\foreach \x in {-5,...,2}
	%{
	%	\node [scale=\y][above] at (\x,-1.5) {$\x$};
	%};
	\node [scale=2][above] at (0+\t,-0.7) {$0$};
	\node [scale=2.5][above] at (-5+\t,0) {$\Circle$};
	\node [scale=2.5][above] at (-4+\t,0) {$\CIRCLE$};
	\node [scale=2.5][above] at (-3+\t,0) {$\Circle$};
	\node [scale=2.5][above] at (-2+\t,0) {$\CIRCLE$};
	\node [scale=2.5][above] at (-1+\t,0) {$\CIRCLE$};
	\node [scale=2.5][above] at (0+\t,0) {$\circledast$};
	\node [scale=2.5][above] at (1+\t,0) {$\CIRCLE$};
	\node [scale=2.5][above] at (2+\t,0) {$\CIRCLE$};
	\node [scale=2.5][above] at (3+\t,0) {$\Circle$};
	\node [scale=2.5][above] at (4+\t,0) {$\Circle$};
	\draw [->,>=stealth] (4,2) -- (5,2);
	\node [scale=2.5][above] at (-5+\t,0) {$\Circle$};
	%second step
	\node [scale=2][above] at (10,-0.7) {$0$};
	\node [scale=2.5][above] at (6,0) {$\CIRCLE$};
	\node [scale=2.5][above] at (7,0) {$\Circle$};
	\node [scale=2.5][above] at (8,0) {$\CIRCLE$};
	\node [scale=2.5][above] at (9,0) {$\CIRCLE$};
	\node [scale=2.5][above] at (10,0) {$\Circle$};
	\node [scale=2.5][above] at (11,0) {$\CIRCLE$};
	\node [scale=2.5][above] at (12,0) {$\CIRCLE$};
	\node [scale=2.5][above] at (13,0) {$\Circle$};
	\node [scale=2.5][above] at (14,0) {$\Circle$};
	\draw[->] (15,0)  to [out=0,in=0, looseness=1] (15,-4);
	%third step
	\node [scale=2][above] at (10,-5.7) {$0$};	
	\node [scale=2.5][above] at (8,-5) {$\CIRCLE$};
	
	\node [scale=2.5][above] at (9,-5) {$\CIRCLE$};
	\node [scale=2.5][above] at (9,-4) {$\CIRCLE$};
	
	\node [scale=2.5][above] at (10,-5) {$\CIRCLE$};
	\node [scale=2.5][above] at (10,-4) {$\CIRCLE$};
	%\node [scale=2.5][above] at (10,-3) {$\circledast$};
	\node [scale=2.5][above] at (11,-5) {$\Circle$};
	\draw [->,>=stealth] (5,-3) -- (4,-3);
	%\forth step
	\node [scale=2][above] at (0+\t,-5.7) {$0$};	
	\node [scale=2.5][above] at (\t-2,-5) {$\CIRCLE$};
	
	\node [scale=2.5][above] at (\t-1,-5) {$\CIRCLE$};
	\node [scale=2.5][above] at (\t-1,-4) {$\CIRCLE$};
	
	\node [scale=2.5][above] at (0+\t,-5) {$\CIRCLE$};
	\node [scale=2.5][above] at (0+\t,-4) {$\CIRCLE$};
	\node [scale=2.5][above] at (0+\t,-3) {$\circledast$};
	\node [scale=2.5][above] at (1+\t,-5) {$\Circle$};
	\end{tikzpicture}
	\caption{\small The four steps in mapping a TASEP configuration with a second-class particle to a TAZRP configuration with a second-class particle.}\label{fig:TZS}
\end{figure}
\subsection{The coupling with second-class particle}\label{ssec:scc} In \cite{G2014}, Gon\c{c}alves generalized the coupling discussed in the preceding subsection and introduced a coupling between the TASEP and the TAZRP where the configurations in both dynamics have one second-class particle. Let $\xi_0\in\mathcal{Y}$ be a TASEP configuration with a second-class particle $q_2$, and $i\in\mathbb{Z}$ be the index of the first hole to the right of $q_2$. Replace the second-class particle $q_2$ with a hole to obtain the configuration $\xi'\in \mathcal{Y}$, so that $\Phi(\xi')=\eta'\in\mathcal{X}$. Finally, put a second-class particle $p_2$ on top of the $i$'th column in $\eta'$ to obtain $\eta_0$. The mapping (Figure \ref{fig:TZS}) just defined is a bijection between TASEP and TAZRP configurations with second-class particles so that throughout the dynamics of $\xi_t$ (TASEP starting from $\xi_0$) and $\eta_t$ (TAZRP starting from $\eta_0$) the position of the second-class particle in the TAZRP can tell the flux of holes seen by the second-class particle in the TASEP. The position, and hence the speed of the second-class particle $p_2$, can be found by considering the flux of holes passing across the second-class particle $q_2$ in the TASEP configuration. Let
\begin{align}
H_{2}^{tasep}(t,\xi_0)=\inf\{i:\text{the hole $y_i$ is to the right of $q_2$}\}
\end{align}
be the number of holes that have crossed the second-class particle $q_2$  under TASEP dynamics  starting from $\xi_0$ (here we assume that the $q_2$ was positioned at time $t=0$ between $y_{-1}$ and $y_{0}$). It is not hard to verify that $X_{p_2}(t)=H^{tasep}_2(t,\xi)$, i.e. the position of the particle $p_2$ equals the flux of holes crossing $q_2$. Let $\xi_0$ be a Riemann initial data; that is, the limits
\begin{align}
\rho:=\lim_{m\rightarrow\infty}\frac 1m\sum_{k=-m}^{-1}\xi_0\left(k\right)\quad\text{and}\quad\lambda:=\lim_{m\rightarrow\infty}\frac 1m\sum_{k=0}^{m}\xi_0\left(k\right)
\end{align}
exist.
It was shown in \cite{F.M.P.}[Proposition 2.2 and Theorem 3] that
\begin{align}\label{sf}
\lim_{t\rightarrow{+\infty}}\frac{H_{2}^{tasep}(t,\xi_0)}{t}=\Big(\frac{1+\mathcal{U}}{2}\Big)^2 \quad \textrm{ almost  surely},
\end{align}
where $\mathcal{U}$ is the speed of the second-class particle $q_2$. It now follows  that the speed of the second-class particle $p_2$ starting from a Riemann initial condition equals $\Big(\frac{1+\mathcal{U}}{2}\Big)^2$. In \cite{F.M.P.}, it was also shown that when the initial configuration is i.i.d. on either side of the origin and $\lambda<\rho$ then  $\mathcal{U}$ has uniform distribution on the interval $[1-2\rho,1-2\lambda]$.  In this paper, we are mostly interested in the case where $\lambda=1$ and $\rho=0$ which corresponds, by the coupling above, to the TAZRP starting with infinitely many particles at site $-1$, a second-class particle at site $0$ and holes on all positive sites. In this case, the distribution of the speed $U$ of $p_2$ is given by
\begin{align}\label{dou}
	&\mathbb{P}(U\leq v)=\sqrt{v}\quad v\in [0,1].
\end{align}
  \\
The coupling described above between configurations with second-class particles can be extended to configurations with finitely many second-class particles up to the point in time where one second-class particle attempts a jump to a site where there is another second-class particle. Let $\xi$ be a TASEP configuration with $m$ second-class particles such that between the positions of any two second-class particles there is at least one hole. This corresponds in the TAZRP to the case where each column has at most one second-class particle. First we register in $\{i_j\}_{j=1}^m$ the indices of the holes to the right of each second-class particle. Then replace the second-class particles with holes, apply the mapping $\Phi$ on the new configuration and finally place second-class particles at the top of columns $i_1,...,i_m$. We will use this coupling in the proof of Theorem \ref{thm:ov}.
\begin{rem}
	In \cite{CP}, in the context of last-passage percolation, Cator and Pimentel obtained the distribution of the speed of a second-class particle in any Riemann initial condition. Using the coupling in Subsection \ref{ssec:CTZ} one can translate the results in \cite{CP} to results for the distribution of the speed of a second-class particle in the TAZRP starting from a larger set of initial conditions.
\end{rem}
Before we turn to the proof of Lemma \ref{Fpro}, we note that it is a straightforward consequence of (\ref{sf}). Nevertheless, for the sake of self-containment we give here a proof that uses only the results in \cite{F.M.P.}, that the second-class particle positioned at the origin between all particles to the left and holes to its right has a speed with probability 1, and that this speed is $>-1$ a.s. .
\begin{proof}[Proof of Lemma \ref{Fpro}]
	$\,$\\\textbf{Step 1:} We first show that
	the particle $p_{0,0}$ (the particle located at the bottom of the
	$0$'th column) has a speed. By calling $p_{0,0}$ a $2$nd class particle, all particles  to its left are 1st class compared to it as they are of higher class. Similarly all particles to its right or above it are seen as  $3$rd class (holes), and using the coupling of the TAZRP with the TASEP we get to the TASEP configuration
	
	\[
	\ldots111123333\ldots.
	\]
	It was shown in \cite{MG,F.M.P.} that particle $2$ has speed $\mathcal{U}$ a.s., and as explained in Subsection \ref{ssec:scc}, the speed $U_{0,0}$ equals
	 \begin{align}
	 	U_{0,0}=\Big(\frac{1+\mathcal{U}}{2}\Big)^2.
	 \end{align}
And in particular it is strictly positive with probability $1$.
	\\
	\textbf{Step 2:} We now claim that
	particle $p_{0,i}$ also develops a strictly positive speed for all $i>0$. Consider
	the event that $i-1$ particles jump from column -1 to 0 before any
	other jump is made in columns -2 and 0. This event has positive probability,
	and if it happens we reach  a configuration where at site $0$
	we have
	\[
	\begin{array}{c}
	\vdots\\
	p_{0,0}\\
	p_{-1,-\left(i-1\right)}\\
	\vdots\\
	p_{-1,0}
	\end{array}.
	\]
	Under this event $p_{0,0}$ sees the same environment as $p_{0,i}$
	sees at $\eta^{*}$ (recall (\ref{initial condition})), that is, all the particles below it or to its left are first class particles while all particles to its right or above it are holes. Thus if $p_{0,i}$ has positive probability not
	to have a strictly positive speed under $\eta^{*}$, then so does $p_{0,0}$, contradicting
	the fact that with probability $1$ $p_{0,0}$ has a strictly positive speed.
\end{proof}
%In this section we start by recalling  a relation  which holds in the one-dimensional case, between the nearest-neighbor simple exclusion and a zero-range process, introduced by
%Kipnis in \cite{K.}.   The totally asymmetric zero-range process that will be mapped with the TASEP, has space state $\mathcal X$ and generator
%defined on local functions by
%\begin{equation*}
%{L}^{ZR} f(\eta)=\sum_{x\in{\mathbb{Z}}}1_{\{\eta(x)\geq{1}\}}[f(\eta^{x,x-1})-f(\eta)],
%\end{equation*}
%where $\eta^{x-1,x}$ is given above.
We end this section with a proof of Lemma \ref{lem:Esum=1}.
\begin{proof}[Proof of Lemma \ref{lem:Esum=1}]
We will use the mass transport principle (See e.g. \cite{PerLy} Chapter $8$).
For each $k\geq 0$ and any $t$ define a mass-transport function $f^k_t : \Z \times \Z \rightarrow [0,\infty]$ by
\[
f^k_t(z,y)=\begin{cases}
\sum_{i=0}^k \mathrm{1}_{X_{z,i}(t)>y} & y\geq z\\
0 & y<z.
\end{cases}
\]
%$f^k_t(z,y)=\sum_{i=0}^k 1_{X_{z,i}>y}$ if $y\geq z$ and $f^k_t(z,y)=0$ if $y<z$.
One may think of this as each of the first $k$ particles in every column sending a mass of $1$ to each position they jump from up to time $t$.
Define $Z^k_t=\frac{1}{t} \sum_{y=-\infty}^{\infty} f^k_t(0,y)$ and $Y^k_t=\frac{1}{t}\sum_{z=-\infty}^{\infty} f^k_t(z,0)$.
$Y^k_t$ is just the averaged rate at which particles of (initial) height $\leq k$ jumped from $0$, and therefore $\mathbb{E}[Y^k_t]\leq 1$ and $\lim_{k\rightarrow \infty} \mathbb{E}[Y^k_t]=1$ for all $t$.
On the other hand, $\lim_{t\rightarrow\infty} Z^k_t=\sum_{i=0}^k U_{0,i}$. Since the distribution of $f^k_t$ is translation invariant, the mass-transport principle gives that $\mathbb{E}[Z^k_t]=\mathbb{E}[Y^k_t]$, and therefore $\sum_{i=0}^\infty \mathbb{E}[U_{i,0}] =1$.

We now turn to $\inf_{i\geq 0}U_{z,i}$. Fix any $\epsilon>0$. Since $U_{z,i}\geq 0$, by Markov's inequality $\PP(U_{z,i})>\epsilon)<\frac{\mathbb{E}[U_{z,i}]}{\epsilon}$, and therefore $\sum_{i\geq 0}\PP(U_{z,i}>\epsilon) < \infty$. 
By the Borel-Cantelli Lemma this a.s. happens only for finitely many $i$-s, and therefore $\PP(\inf_{i\geq 0}U_{z,i}\leq \epsilon)=1$. Since $\epsilon$ was arbitrary, we are done.
\end{proof}

\section{stationarity of the TAZRP speed process}\label{sec:odd}

The key to understanding why the distribution of $U$
(or more precisely, its reflection)
gives a stationary distribution for the multi-type TAZRP
is to understand the effect of a small change
to the initial condition $\eta^{*}$ on the speed process.

Specifically, in the next lemma we consider how the speed process starting from $\sigma_0\eta^*$ is different from that starting from $\eta^*$, where $\sigma_x$ is the operator given in (\ref{sig}) and $\eta^*$ is the initial condition given in (\ref{initial condition}). More precisely,  consider the two initial conditions $\eta(0)=\eta^*$ and   $\eta'(0)=\sigma_0\eta^*$. Let $\mathcal{T}$ be a Poisson process on $\mathbb{Z}\times\mathbb{R}_{+}$, representing the different clocks on the sites of $\mathbb{Z}$. Apply now the Harris construction with $\mathcal{T}$ to the two initial conditions  $\eta(0)$ and $\eta'(0)$. Let $U$ be the speed process associated with the process $\eta$ as defined in Definition \ref{def}. We define the speed process $U'$ to be the speed process associated with the process $\eta'(t)$, that is
\begin{align}
	U'_{z,i}=\lim_{t\rightarrow\infty}t^{-1}X_{z,i}(t),
\end{align}
 where we assume the dynamics starts from $\eta'$.
It is important to note that particles in $\eta(0)$ and $\eta'(0)$ are the same particles indexed by $\mathbb{Z}\times \mathbb{N}$ according the their position in $\eta(0)$. More precisely, the particle $\eta^*(0,0)$ is identified with $(0,0)$ and its class is the number $\eta^*(0,0)$. The operation $\sigma_0$ will move the particle $(0,0)$ and place it at the bottom of column $1$ (this is due to the initial order imposed on $\eta^*$ where every particle in the $0$'th column is stronger than any particle in column $1$). However it is important to note that we still identify this particle by its initial position in $\eta^*$, i.e. $(0,0)$. This means that the position of the particle $(0,0)$ in $\eta^*$ is $X_{0,0}(0)=0$ whereas its position in $\sigma_0\eta^*$ is $X'_{0,0}(0)=1$. Hence, the  arrays $\{U_{z,i}\}_{(z,i)\in\mathbb{Z}\times \mathbb{N}_0}$ and $\{U_{z,i}'\}_{(z,i)\in\mathbb{Z}\times \mathbb{N}_0}$ register the speed of particle $(0,0)$ at position $(0,0)$ of the array, despite the fact that $X'_{0,0}(0)=1$.
\begin{lem}
\label{lem:effect_jump}Let $\eta$ and $\eta'$ be two TAZRPs
defined by the Harris construction with initial condition $\eta^{*}$and
$\sigma_{0}\eta^{*}$ respectively, and a Poisson process $\mathcal{T}$ on $\mathbb{Z}\times\mathbb{R}_{+}$.
Let $U$ and $U'$ be the TAZRP speed processes associated with $\eta$
and $\eta'$ respectively, then
\begin{equation}\label{eq: the effect of a jump-2}
\sigma_{0}^{*}U=U'.
\end{equation}
\end{lem}

In order to prove
the lemma we shall make use of  a process we call the sorting process. The configuration $\overline{\eta}$ of the sorting process compares two initial configurations $\eta$ and $\xi$ of the multi-type TAZRP. Applying the dynamics of the sorting process on the initial configuration keeps track of the development of the two initial configurations $\eta$ and $\xi$ when one applies on them the same Poisson clocks in the Harris construction.\\ For $\left(x_{1},y_{1}\right),\left(x_{2},y_{2}\right)\in\mathbb{R}^{2}$
we write $\left(x_{2},y_{2}\right)\leq\left(x_{1},y_{1}\right)$ whenever
$x_{1}\geq x_{2}$ and $y_{1}\geq y_{2}$. Let
\begin{align}
	\mathcal{W}=\left\{ \overline{\eta}\in\left(\mathbb{R}^{2}\right)^{\mathbb{Z}\times\mathbb{N}_{0}}:\overline{\eta}_{z,i}\leq\overline{\eta}_{z,j},\text{\text{ for all \ensuremath{j\leq i} and \ensuremath{z}}}\right\}.
\end{align}
Note that $\mathcal{W}$ is simply an array indexed by $\mathbb{Z}\times\mathbb{N}_{0}$  that contains pairs of real numbers. We say that $\left(x_{1},y_{1}\right),\left(x_{2},y_{2}\right)$ are
ordered if either $\left(x_{1},y_{1}\right)\leq\left(x_{2},y_{2}\right)$
or $\left(x_{1},y_{1}\right)\geq\left(x_{2},y_{2}\right)$, otherwise
we say that they are unordered. Let $\overline{\eta}\in \mathcal{W}$, and for $k\in\{1,2\}$ let $\overline{\eta}_{z,i}^k$ denote the $k$'th component of the pair $\overline{\eta}_{z,i}$. We attach independent Poisson clocks of rate 1 to
each site (column) $x\in \mathbb{Z}$, at the ring of the clock of the column $x$, the largest (with respect to the order on pairs)
pair sitting at the bottom of the column, jumps to the column to its right where the pairs rearrange into
elementwise order. More precisely, if the pair $\overline{\eta}_{z,0}$
jumps to column $z+1$, then we arrange the sets of numbers
\begin{align}
	A&=\overline{\eta}_{z,0}^1\cup\left\{ \overline{\eta}_{z+1,i}^1\right\} _{i\in\mathbb{N}_{0}}\\
	B&=\overline{\eta}_{z,0}^2\cup\left\{ \overline{\eta}_{z+1,i}^2\right\} _{i\in\mathbb{N}_{0}},\nonumber
\end{align}
according to their order to obtain the decreasing sequences $\left\{ a_{i}\right\} _{i=0}^{\infty}$
and $\left\{ b_{i}\right\} _{i=0}^{\infty}$. Then replace the column
$\overline{\eta}_{z+1,\cdot}$ by a new column whose $i$'th element
is $\left(a_{i},b_{i}\right)$. We call this process on $\mathcal{W}$
the \emph{sorting process.} We say the pairs $\left(x_{1},y_{1}\right)=\overline{\eta}_{z,0}\left(t-\right)$
and $\left(x_{2},y_{2}\right)\in\overline{\eta}_{z+1,\cdot}\left(t-\right)$
interact if the jump of the pair $\left(x_{1},y_{1}\right)$ at time
$t$ to column $z+1$ results in $\left(x_{2},y_{2}\right)\notin\overline{\eta}_{z+1,\cdot}\left(t\right)$.
Note that $\left(x_{1},y_{1}\right)$ interacts only with pairs in
$\overline{\eta}_{z+1,\cdot}\left(t-\right)$ that are unordered with
respect to itself (see Figure \ref{fig:sp}). We make the following observations:
\begin{enumerate}
\item If $\left(x,y\right)$ is a pair in $\overline{\eta}$ that is ordered
with respect to all other pairs in $\overline{\eta}$, then $\left(x,y\right)$
will not interact throughout the dynamics.
\item If $\eta,\xi\in\mathcal{Z}$, then $\overline{\eta}_{z,i}=\left(\eta(z,i),\xi(z,i)\right)\in\mathcal{W}$.
\end{enumerate}
\begin{figure}[t!]
	\centering%
	\begin{tikzpicture}[scale=0.4, every node/.style={transform shape}]
	\node [scale=2][above][red] at (0,0) {$(5,7)$};
	\node [scale=2][above] at (0,2) {$(3,4)$};
	\node [scale=2][above] at (0,4) {$(2,0)$};
	\node [scale=2][above] at (0,5.5) {$\vdots$};
	%next column
	\node [scale=2][above] at (3,0) {$(10,8)$};
	\node [scale=2][above][red] at (3,2) {$(4,8)$};
	\node [scale=2][above] at (3,4) {$(3,1)$};
	\node [scale=2][above] at (3,5.5) {$\vdots$};
	\draw[->] (0,1)  to [out=60,in=160, looseness=1] (2.3,1.7);
	%next column
	\draw [->,>=stealth] (5,2.5) -- (6,2.5);
	%next column
	\node [scale=2][above] at (8,0) {$(3,4)$};
	\node [scale=2][above] at (8,2) {$(2,0)$};
	\node [scale=2][above] at (8,3.5) {$\vdots$};
	%next column
	\node [scale=2][above] at (11,0) {$(10,8)$};
	\node [scale=2][above][red] at (11,2) {$(5,8)$};
	\node [scale=2][above][red] at (11,4) {$(4,7)$};
	\node [scale=2][above] at (11,6) {$(3,1)$};
	\node [scale=2][above] at (11,7.5) {$\vdots$};
	\end{tikzpicture}
	\caption{\small Illustration of the sorting process dynamics. The particle $(5,7)$ interacts with the only particle that is not ordered with respect to, $(4,8)$, to create two new particles - $(4,7)$ and $(5,8)$.}\label{fig:sp}
\end{figure}
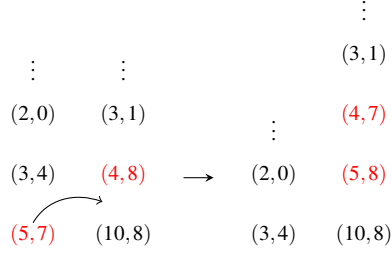

\begin{proof}[Proof of Lemma \ref{lem:effect_jump}]
Define $\overline{\eta}$ by $\overline{\eta}_{z,i}=\left(\eta(z,i)^{*},\sigma_{0}\eta(z,i)^{*}\right)$, where $\eta^*$ is as in (\ref{initial condition}),
and let $\overline{\eta}_{z,i}\left(t\right)$ be the sorting process
starting from the initial condition $\overline{\eta}$. The idea of the proof is that the sorting process marginals $\overline{\eta}^1$ and $\overline{\eta}^2$ are the multi-type TAZRP with initial conditions $\eta^{*}$ and $\sigma_{0}\eta^{*}$ respectively, this allows us to compare the position of the same particle in the two processes. First note
that all pairs in $\overline{\eta}$ are ordered with respect to any
other pair except the pairs in $\overline{\eta}_{0,\cdot}(0)$ and the
pair $\overline{\eta}_{1,0}(0)$. It follows that pairs that are not
in $\overline{\eta}_{0,\cdot}(0)\cup\overline{\eta}_{1,0}(0)$ do not interact
throughout the dynamics. Let
\begin{align}
	A=\left\{ i:\overline{\eta}_{0,i}(0)\text{ interacts with a pair \ensuremath{\left(x,y\right)} s.t. \ensuremath{x=\eta^*(1,0)}}\right\},
\end{align}
and let
\[
i_{fast}=\begin{cases}
-1, &A=\emptyset\\
\sup A, &A\neq\emptyset.
\end{cases}
\]
Note that if $U_{0,i}<U_{1,0}$ then $i_{fast}<i$ as particle $p_{0,i}$ cannot overtake particle $p_{1,0}$ and so $\overline{\eta}_{0,i}(0)$ cannot interact with any pair of the form $(\eta^*_{1,0},y)$ for some $y\in \eta^*_{0,\cdot}$. As $\lim_{i\rightarrow\infty}U_{0,i}=0$
we conclude that $i_{fast}\neq \infty$. On the event that $i_{fast}\geq 0$, the pairs $\left\{ \overline{\eta}_{0,i}\right\} _{i=0}^{i_{fast}}$
will interact with particles whose first coordinate is $p_{1,0}$
according to their order. Once the pair $\overline{\eta}_{0,0}$ has
interacted with $\overline{\eta}_{1,0}$ at some time $t_{0}>0$,
then the two pairs are ordered into two new pairs $\left(p_{0,0},p_{0,0}\right)$
and $\left(p_{1,0},p_{0,1}\right)$. The pair $\left(p_{0,0},p_{0,0}\right)$
is ordered w.r.t. all pairs in $\overline{\eta}\left(t_{0}\right)$ and
therefore will not interact at later times $t>t_0$. The next interaction (if $i_{fast}>0$)
will be between the pairs $\overline{\eta}_{0,1}=\left(p_{0,1},p_{0,2}\right)$
and $\left(p_{1,0},p_{0,1}\right)$ at some time $t_{1}>t_{0}$. The
interaction will lead to the formation of the pairs $\left(p_{0,1},p_{0,1}\right)$
and $\left(p_{1,0},p_{0,2}\right)$ at time $t_{1}$ and we see that
the pair $\left(p_{0,1},p_{0,1}\right)$ is ordered w.r.t. all other
pairs and so will not interact again (see Figure \ref{fig:points2}). We continue in the same way
until all pairs $\left\{ \left(p_{0,i},p_{0,i}\right)\right\} _{i=0}^{i_{fast}}$
have formed by time $t_{i_{fast}}$ as well as the pair $\left(p_{1,0},p_{0,i_{fast}+1}\right)$.
By  the definition of $i_{fast}$, no interactions will occur
at time $t>t_{i_{fast}}$. Now, let  $X_t$ and $X'_t$ be the processes that keep track  of the horizontal position of the different particles in $\eta$ and $\eta'$ respectively i.e.
\begin{align}
	X_{z,i}(t)=n &\iff p_{z,i}\in \eta(n,\cdot)(t)\\
	X'_{z,i}(t)=n &\iff p'_{z,i}\in \eta'(n,\cdot)(t).\nonumber
\end{align}
This implies that for $t>t_{i_{fast}}$
\begin{align}
	X_{z,i}(t)&=X'_{z,i}(t)\quad \text{if} \quad z\notin \{0,1\} \vee (z=0,0\leq i\leq i_{fast}) \\\nonumber
	X_{z,i+1}(t)&=X'_{z,i}(t) \quad \text{if} \quad z=1,i>0\\\nonumber
	X_{0,i}&=X'_{0,i+1}\quad \text{if} \quad i>i_{fast}+1\\
	X_{0,i_{fast}+1}(t)&=X'_{1,0}(t)\nonumber.
\end{align}
Multiplying by $t^{-1}$ and letting $t$ go to infinity we  obtain
\begin{align}\label{rel}
	U_{z,i}&=U'_{z,i}\quad \text{if} \quad z\notin \{0,1\} \vee (z=0,0\leq i\leq i_{fast}) \\\nonumber
U_{z,i+1}&=U'_{z,i} \quad \text{if} \quad z=1,i>0\\\nonumber
U_{0,i}&=U'_{0,i+1}\quad \text{if} \quad i>i_{fast}+1\\
U_{0,i_{fast}+1}&=U'_{1,0}\nonumber.
\end{align}
One can now verify that the relations in (\ref{rel}) between $U$ and $U'$ as arrays indexed by $\mathbb{Z}\times \mathbb{N}_0$ are equivalent to (\ref{eq: the effect of a jump-2}) as configurations in $\mathcal{W}$, and the result is proved (That $i_{fast}+1=i_{sort^{0}}$ is a consequence of Theorem \ref{thm:ov}, but  we do not need it here).
\end{proof}
%%%%% figure for the sorting process
\begin{figure}[ht!]
	
\end{figure}
\begin{figure}[ht!]
	\centering%
		\begin{subfigure}[t]{{.9\textwidth}}
		\centering
		\centering%
		\def\y{2.2}
		\def\z{5}
		\begin{tikzpicture}[scale=0.4, every node/.style={transform shape}]
		\foreach \x in {-1,...,2}
		{
			\node [scale=\y][above] at (\x*\z,-1.5) {$\x$};
		};
		\node [scale=\y][above] at (-\z-2,4) {$\dots$};
		\node [scale=\y][above] at (-\z,0) {$(p_{-1,0},p_{-1,0})$};
		\node [scale=\y][above] at (-\z,2) {$(p_{-1,1},p_{-1,1})$};
		\node [scale=\y][above] at (-\z,4) {$\vdots$};
		\node [scale=\y][above] at (-\z,6) {$(p_{-1,i},p_{-1,i})$};
		\node [scale=\y][above] at (-\z,8) {$\vdots$};
		%next column
		\node [scale=\y][above][red] at (0,0) {$(p_{0,0},p_{0,1})$};
		\node [scale=\y][above][red] at (0,2) {$(p_{0,1},p_{0,2})$};
		\node [scale=\y][above][red] at (0,4) {$\vdots$};
		\node [scale=\y][above][red] at (0,6) {$(p_{0,i},p_{0,i+1})$};
		\node [scale=\y][above][red] at (0,8) {$\vdots$};
		%next column
		\node [scale=\y][above][red] at (\z,0) {$(p_{1,0},p_{0,0})$};
		\node [scale=\y][above] at (\z,2) {$(p_{1,1},p_{1,0})$};
		\node [scale=\y][above] at (\z,4) {$\vdots$};
		\node [scale=\y][above] at (\z,6) {$(p_{1,i},p_{1,i-1})$};
		\node [scale=\y][above] at (\z,8) {$\vdots$};
		%next column
		\node [scale=\y][above] at (\z*2,0) {$(p_{2,0},p_{2,0})$};
		\node [scale=\y][above] at (\z*2,2) {$(p_{2,1},p_{2,1})$};
		\node [scale=\y][above] at (\z*2,4) {$\vdots$};
		\node [scale=\y][above] at (\z*2,6) {$(p_{2,i},p_{2,i})$};
		\node [scale=\y][above] at (\z*2,8) {$\vdots$};
		\node [scale=\y][above] at (\z*2+2,4) {$\dots$};
		\end{tikzpicture}
		\caption{\small The initial configuration $\overline{\eta}$. Only the pairs in red are not ordered. Any other couple in the configuration is ordered with respect to all other pairs.}
	\end{subfigure}

	\begin{subfigure}[t]{.9\textwidth}
		\centering
		\def\y{2.2}
		\def\z{4}
		\begin{tikzpicture}[scale=0.4, every node/.style={transform shape}]
		\foreach \x in {0,...,2}
		{
			\node [scale=\y][above] at (\x*\z,-1.5) {$\x$};
		};
		%next column
		\node [scale=\y][above][red] at (0,0) {$(p_{0,0},p_{0,1})$};
		\node [scale=\y][above][red] at (0,2) {$(p_{0,1},p_{0,2})$};
		\node [scale=\y][above][red] at (0,4) {$\vdots$};
		\node [scale=\y][above][red] at (0,6) {$(p_{0,i},p_{0,i+1})$};
		\node [scale=\y][above][red] at (0,8) {$\vdots$};
		%next column
		\node [scale=\y][above][red] at (\z,0) {$(p_{1,0},p_{0,0})$};
		\node [scale=\y][above] at (\z,2) {$(p_{1,1},p_{1,0})$};
		\node [scale=\y][above] at (\z,4) {$\vdots$};
		\node [scale=\y][above] at (\z,6) {$(p_{1,i},p_{1,i-1})$};
		\node [scale=\y][above] at (\z,8) {$\vdots$};
		%next column
		\node [scale=\y][above] at (\z*2,0) {$(p_{2,0},p_{2,0})$};
		\node [scale=\y][above] at (\z*2,2) {$(p_{2,1},p_{2,1})$};
		\node [scale=\y][above] at (\z*2,4) {$\vdots$};
		\node [scale=\y][above] at (\z*2,6) {$(p_{2,i},p_{2,i})$};
		\node [scale=\y][above] at (\z*2,8) {$\vdots$};
		\draw[->] (0,0.95)  to [out=60,in=160, looseness=1] (\z-1.2,1.6);
		%second step
		\foreach \x in {0,...,2}
		{
			\node [scale=\y][above] at (\x*\z,-1.5) {$\x$};
		};
		\def\w{13}
		\foreach \x in {0,...,2}
		{
			\node [scale=\y][above] at (\x*\z+\w,-1.5) {$\x$};
		};
		%next column
		\node [scale=\y][above][red] at (\w,0) {$(p_{0,0},p_{0,1})$};
		\node [scale=\y][above][red] at (\w,2) {$(p_{0,1},p_{0,2})$};
		\node [scale=\y][above][red] at (\w,4) {$\vdots$};
		\node [scale=\y][above][red] at (\w,6) {$(p_{0,i},p_{0,i+1})$};
		\node [scale=\y][above][red] at (\w,8) {$\vdots$};
		%next column
		\node [scale=\y][above] at (\z+\w,0) {$(p_{0,0},p_{0,0})$};
		\node [scale=\y][above][red] at (\z+\w,2) {$(p_{1,0},p_{0,1})$};
		\node [scale=\y][above] at (\z+\w,4) {$\vdots$};
		\node [scale=\y][above] at (\z+\w,6) {$(p_{1,i},p_{1,i-1})$};
		\node [scale=\y][above] at (\z+\w,8) {$\vdots$};
		%next column
		\node [scale=\y][above] at (\z*2+\w,0) {$(p_{2,0},p_{2,0})$};
		\node [scale=\y][above] at (\z*2+\w,2) {$(p_{2,1},p_{2,1})$};
		\node [scale=\y][above] at (\z*2+\w,4) {$\vdots$};
		\node [scale=\y][above] at (\z*2+\w,6) {$(p_{2,i},p_{2,i})$};
		\node [scale=\y][above] at (\z*2+\w,8.7) {$\vdots$};
		\draw [->,>=stealth] (\z*2+2,4) -- (\z*2+3,4);
		\end{tikzpicture}
		\caption{\small One step in the sorting process starting from $\overline{\eta}$. The pairs $(p_{0,0},p_{0,1})$ and $(p_{1,0},p_{0,0})$ interact and give rise to two new pairs in column $1$ - $(p_{0,0},p_{0,0})$, which is ordered with respect to any other particle in the configuration, and $(p_{1,0},p_{0,1})$ which is unordered with respect to any pair in column $0$.}
	\end{subfigure}%
	\caption{\small The sorting process.}
	\label{fig:points2}
\end{figure}
%where $\overline{\eta}_{z,i}\left(t\right)\left(i\right)$ is the
%$i$'th element of the vector $\overline{\eta}_{z,i}\left(t\right)$.
%We conclude that
%\begin{equation}
%\left(p_{z,i},p_{z',i'}\right)\in\overline{\eta}\left(t\right)\text{ for large enough \ensuremath{t}}\Rightarrow U_{z,i}=U_{z',i'}^{'}.\label{eq: the effect of a jump}
%\end{equation}
%Next note that by \lemref{the_effect_of_a_jump} we have
%\begin{align}
%i & >i_{fast}\Rightarrow U_{0,i}\leq U_{1,0}\label{eq: the effect of a jump-1}\\
%i & \leq i_{fast}\Rightarrow U_{0,i}\geq U_{1,0}.\nonumber
%\end{align}
%As $i_{fast}=i_{sort^{x}}\left(U_{1,0}\right)$, considering (\ref{lem:the_effect_of_a_jump})
%and (\ref{eq: the effect of a jump-1}) we conclude (\ref{eq: the effect of a jump-2}).

%\begin{figure}[t!]
%
%\end{figure}
We are now ready for the proof of Theorem \ref{thm:sd}. We defer the proof of the uniqueness of $\mu^\pi$ to Section \ref{ms}.
\begin{proof}[Proof of Theorem \ref{thm:sd} without uniqueness]
Let $\mathcal{T}_{0}$ be a Poisson process on $\mathbb{Z}\times\mathbb{R}$
with rate $1$, and let $\mathcal{T}_{s}=\mathcal{T}_{0}+0\times\left(0,s\right)$
be the translation of $\mathcal{T}_{0}$ by $s$ units of time. Also
define $\mathcal{T}_{s}^{+}=\mathcal{T}_{s}\cap\mathbb{Z}\times\mathbb{R}_{+}$,
the restriction of $\mathcal{T}_{s}$ to $\mathbb{Z}\times\mathbb{R}_{+}$.
Define $U\left(s\right)$ to be the speed process constructed through
the Harris construction with initial condition $\eta^{*}$ and the Poisson
process $\mathcal{T}_{s}^{+}$. For each $s>0$, $U\left(s\right)$
has distribution $\mu$, and it is enough to show that $U\left(s\right)$
satisfies the TAZRP dynamics. Starting from $U_{0}$, adding an infinitesimal
time $s$ adds, at each site $i$, the operator $\sigma_{i}$ at rate
$1$. According to \lemref{effect_jump} this should result
in applying $\sigma_{i}^{*}$ to $U\left(0\right)$ to obtain $U\left(s\right)$
at rate one. It is straightforward to see that $\pi\sigma_{i}^{*}\eta=\sigma_{-i-1}\pi\eta$
which implies that the process $\pi U\left(s\right)$ is defined through
the generator (\ref{gem}) and the initial condition
$\pi U_{0}=\mu^{\pi}$ which is exactly what we need. To see that
$\mu^{\pi}$ is ergodic, it is enough to note that $\mu^{\pi}$ is
generated by applying some deterministic mapping $\fu$ on the Poisson
process $\mathcal{T}_{0}$, which is ergodic w.r.t. the translation operator $\tau$, and that
$\tau \fu\left(\mathcal{T}_{0}\right)=\fu\left(\tau\mathcal{T}_{0}\right)$.
\end{proof}
%One can relate the TAZRP speed process to the $n$- type TAZRP. The $n$-type TAZRP is the analogue of the TAZRP with a second-class particle to $n$ classes. At each column, one can find as many as $n$ different types of particles of different priority. When the clock rings for the $x$'th column, the particle of the highest priority jumps to the column $x+1$.   Let
%$\mathcal{Z}_{n}=\mathcal{Z}\cap\mathcal{R}_{n}^{\mathbb{Z}\times\mathbb{N}_{0}}$,
%where $\mathcal{R}_{n}=\left\{ -1,..,-n\right\} $. Let $\eta\left(t\right)$
%be the multi-type TAZRP on $\mathcal{Z}$. Then the set $\mathcal{Z}_{n}$
%is closed under the dynamics of $\eta\left(t\right)$ , that is, if
%$\eta\left(0\right)\in\mathcal{Z}_{n}$ then $\eta\left(t\right)\in\mathcal{Z}_{n}$
%for all $t>0$ . We define the $n$- type TAZRP to be the multi-type
%TAZRP restricted on $\mathcal{Z}_{n}$. Note that the choice of $\mathcal{R}_{n}$
%is not crucial, one can take any ordered set of size $n$. However, the advantage of using the set $\mathcal{R}_{n}$ is that one can  read off the class of the particle by removing the minus sign in the number. Note that the TAZRP is the $n$-type TAZRP for $n=2$, as holes(no particles) can always be thought of as the particle of the lowest class.\\
	Let $\fu:\mathbb{R}\rightarrow\mathbb{R}$ be a non-decreasing
function. Let $\eta\in\mathcal{Z}$, we write $\fu\left(\eta\right)$
for the configuration $\fu\left(\eta\right)_{z,i}=\fu\left(\eta_{z,i}\right)$.
Note that $\fu\left(\eta\right)\in\mathcal{Z}$. An easy yet important observation is that the dynamics of the multi-type TAZRP
(and likewise the TASEP) are conserved under a monotone relabelling
of the types.

\begin{lem}\label{lem:preservence_of_monotonicity} Let $\fu:\mathbb{R}\rightarrow\mathbb{R}$
	be a non-decreasing function. Let $\eta\in\mathcal{Z}$, and let $\mathcal{T}$
	be a Poisson process on $\mathbb{Z}\times\mathbb{R}_{+}$ and consider
	$\eta\left(t\right)$ and $\eta_{\fu}\left(t\right)$, the multi-type
	TAZRP defined through the Harris construction with $\mathcal{T}$
	and the initial conditions $\eta$ and $\fu\left(\eta\right)$ respectively.
	Then
	\begin{align}\label{invariance under \fu}
	\fu\left(\eta\left(t\right)\right)=\eta_{\fu}\left(t\right)\qquad\forall t\geq0.
	\end{align}
	
	%In particular, $F\left(\eta\left(t\right)\right)$ is a multi-type
	%TAZRP.
\end{lem}

\begin{proof}
	By the definition of $\eta_{\fu}$, (\ref{invariance under \fu}) holds for $t=0$. Now, following the Harris construction, it is enough to show, that
	\begin{align}
	\sigma_{i}\fu\left(\eta\right)=\fu\left(\sigma_{i}\eta\right)\qquad \text{for every $i\in\mathbb{Z}$},
	\end{align}
	which is not hard to verify.
\end{proof}
\begin{cor}
\label{col:SIM}Let $\fu:\mathbb{R}\rightarrow\mathcal{R}_{n}$
be an increasing  function. Then the distribution of the
process $\fu\left(\pi U\left(\cdot\right)\right)$ is a stationary and
ergodic distribution for the $n$- type TAZRP.
\end{cor}
\begin{proof}
Since $\fu$ is increasing, $\fu\left(\pi U\left(t\right)\right)\in\mathcal{Z}_{n}$
for every $t>0$. By \lemref{preservence_of_monotonicity}, the stationarity
of $\pi U\left(\cdot\right)$ implies the stationarity of $\fu\left(\pi U\left(\cdot\right)\right)$.
Moreover, since $\tau \fu\left(\pi U\left(\cdot\right)\right)=\fu\left(\tau\pi U\left(\cdot\right)\right)$
we see that the ergodicity of $U\left(\cdot\right)$ implies that
of $\fu\left(\pi U\left(\cdot\right)\right)$.
\end{proof}
One can use the one-point marginals of the $1$- type TAZRP along
with \colref{SIM} to obtain the one-point
marginal of $U$.
\begin{lem}\label{lem:opm}
Let $U$ be the TAZRP speed process.
Then, for every $j\in\mathbb{N}_{0}$
\begin{equation}\label{eq:one_point_marginal_TAZRP-2}
\mathbb{P}\left(U_{0,j}\leq v\right)=1-\left(1-\sqrt{v}\right)^{j+1} \quad v\in[0,1].
\end{equation}
\end{lem}
\begin{proof}
Let
\begin{equation}\label{eq:\fu_v}
\fu_{v}\left(x\right)=\left\{ \begin{array}{cc}
-1 & x>v\\
-2 & x\leq v
\end{array}\right..
\end{equation}
Note that by \colref{SIM}, $\fu_{v}\left(U\right)$
is a stationary and ergodic measure of the $1$-type TAZRP and that
\begin{align*}
\mathbb{P}\left(U_{0,0}\leq v\right) & =\mathbb{P}\left(\#\left\{ i:\fu\left(U\right)_{0,i}=-1\right\} =0\right)\\
 & =\mathbb{P}_{\mu_{\alpha}}\left(\eta_{0}=0\right)\\
 & =\frac{1}{1+\alpha},
\end{align*}
where in the second equality we used the well-known unique stationary ergodic measures for  the TAZRP mentioned in (\ref{equilibrium measures for TAZRP}). By (\ref{dou}) we see that $\mathbb{P}\left(U_{0,0}\leq v\right)=\sqrt{v}$,
and therefore that
\begin{equation}
\alpha=\frac{1-\sqrt{v}}{\sqrt{v}}.\label{eq:one_point_marginal_TAZRP}
\end{equation}
Similarly, we see that
\begin{align}
\mathbb{P}\left(U_{0,j}\leq v\right) & =\mathbb{P}_{\mu_{\alpha}}\left(\eta_{0}\leq j\right)\label{eq:one_point_marginal_TAZRP-1}\\
 & =1-\left(\frac{\alpha}{1+\alpha}\right)^{j+1}.\nonumber
\end{align}
Plugging (\ref{eq:one_point_marginal_TAZRP}) in \eqref{one_point_marginal_TAZRP-1}
we obtain the result.
\end{proof}
\begin{rem}
	Lemma \ref{lem:opm} implies the result in \cite{G2014}[Theorem 2.1, case $\rho=\infty$]. Indeed, the equality there can be written with our notation and by using the monotonicity of the speeds of particles in one column, as
	\begin{align}
		\lim_{t\rightarrow \infty}\sum_{i=0}^\infty \mathbb{P}(X_{0,i}(t)\geq ut)= \frac{1-\sqrt{u}}{\sqrt{u}},\nonumber
	\end{align}
	which follows easily by using (\ref{eq:one_point_marginal_TAZRP-2}).
\end{rem}

\section{stationary measures for the $n$-type TAZRP}\label{sec:sq}

\subsection{One-column distribution in stationarity}

Our approach to
investigating the $n$-type TAZRP  is through thinking of each column of the
$n$-type TAZRP as a queue.
Such a queue has services at times of a Poisson process
of rate $1$, and its arrival process contains particles
of types from $1$ to $n$.
The server attends to particles according to their class;
when a service occurs, the particle with the
highest priority is served (if any particle is present), and departs from the queue.

Let $\lambda_i$ be the intensity of arrivals of type $i$.
We are interested in the case where the behaviour of the
queue is stationary in time and ergodic, with a finite
average number of particles of each type present in the queue,
and so we need
\begin{align}\label{lc}
\sum_{i=1}^{n}\lambda_{i}<1.
\end{align}
We wish to consider stationary distributions
of the $n$-type TAZRP which are translation-invariant.
In this case the departure process from the queue
(say, the process of particles moving from site $x$ to site
$x+1$) has the same distribution as the arrival process to
the queue (say, the process of particles moving from
site $x-1$ to site $x$).
In this sense we say that the distribution of the arrival
process is a fixed point for the queueing server.
Using a coupling approach analogous to that
used by Mountford and Prabhakar \cite{MP95},
one can show that for any $\lambda_1, \dots, \lambda_n$
satisfying (\ref{lc}),
there is a unique ergodic fixed point with
intensity $\lambda_i$ of arrivals of type $i$
(see \cite{martin2010fixed} for discussion).
We denote this process by $F^{(n)}$, or
$F^{(n)}_{\lambda_1,\dots,\lambda_n}$ when we need
to emphasise the dependence on the arrival intensities.

%
%
%
%More precisely, we consider $F^{(n)}$ to be a marked point process on $\mathbb{R}$ i.e.  $(t,i)\in F^{(n)}$ denotes that particle of type $i\in \{1,...,n\}$ arrived to the queue at time $t$. The
%queue has one server with rate $1$. The server attends to the particles
%according to their class, that is, the particle of the highest class
%in the queue is attended first. In the multi-type TAZRP setting (\ref{gem}), this would amount to setting the value of customer $i$ to be $-i$. Using similar ideas to the ones in \cite{MP95}, one can show that for any $0<\lambda_1,...,0<\lambda_n$ such that (\ref{lc}) holds, there exits a unique $F^{(n)}$ that is ergodic and stationary with respect to the queue's dynamics. By our assumption of stationarity
%we conclude that the process of arrivals to the queue, $F^{\left(n\right)}$,
%has the same distribution as the process of departure.

Let us mention a few immediate properties of the
processes $F^{(n)}$:
\begin{itemize}
\item
By Burke's Theorem, the process $F^{(1)}_{\lambda_1}$ is a
Poisson process of rate $\lambda_1$.
\item
More generally, again by Burke's Theorem, for each $i$, the combined process of all points in $F^{(n)}_{\lambda_1, \dots, \lambda_n}$ of types $1,\dots, i$ is a Poisson process with rate $\sum_{j=1}^{i}\lambda_{i}$.
\item
The process $F^{(n)}_{\lambda_1,\dots,\lambda_n}$ restricted to types $1,\dots, n-1$,
i.e.\ removing the type-$n$ points, gives the process
$F^{(n-1)}_{\lambda_1,\dots,\lambda_{n-1}}$.
\end{itemize}

The following proposition, which is
the starting-point of our analysis of $n$-type equilibrium
distributions,
shows that $F^{\left(n\right)}$ can be obtained by feeding $F^{\left(n-1\right)}$ into a queue with service rate $\sum_{j=1}^{n}\lambda_{i}$.
It was shown as a by-product of the construction of the multi-type
Hammersley process by Ferrari and Martin in \cite{ferrari2009multiclass}, and more directly
using interchangeability properties of queues by Martin and Prabhakar in \cite{martin2010fixed}.
\begin{prop}\label{prop:ntn}
	Consider an exponential server with rate $\sum_{i=1}^{n}\lambda_{i}$,
	and an arrival process with distribution $F^{(n-1)}$.
	Take the departure process and add to it a point of type $n$ whenever
	the queue has an unused service. The resulting output process has
	distribution $F^{\left(n\right)}$.
\end{prop}
For $0<s\leq 1$ we write $\mathbb{P}^{\left(s\right)}_{\lambda_{1},...,\lambda_{n}}$ for the distribution of the
vector $\left(Q_{1},...,Q_{n}\right)$, where $Q_{i}$ is the number
of particles of type $i$ at some fixed time in the queue with arrival process $F^{\left(n\right)}_{\lambda_{1},...,\lambda_{n}}$
and with an exponential server of rate $s$. Where there is no room for confusion we abbreviate by $\mathbb{P}^{(s)}$.
%We shall also use the distribution $\mathbb{P}^{\left(\lambda_{1},...,\lambda_{i}\right)}$
%of the vector $\left(Q_{1},...,Q_{i-1}\right)$ in a queue fed with
%$F^{\left(i-1\right)}$ served at rate $\sum_{j=1}^{i}\lambda_{j}$.
\begin{rem}\label{rem}
	For a fixed $1 \leq i\leq n$, let $c=\sum_{j=1}^{i}\lambda_{j}$. Note that the distribution of $\mathbb{P}^{\left(c\right)}_{\lambda_{1},...,\lambda_{i}}$
	is equal to that of $\mathbb{P}^{\left(1\right)}_{\nicefrac{\lambda_{1}}{c},...,\nicefrac{\lambda_{n}}{c}}$ restricted on $\left(Q_{1},...,Q_{i-1}\right)$.
\end{rem}
\begin{proof}[Proof of Theorem \ref{thm:mq}]
To prove Theorem \ref{thm:mq}, we
need to show that under $\PP^{(1)}$, the
distribution of $Q_1,\dots, Q_n$ is that given by $(\ref{Qd})$.
%The existence, uniqueness, stationarity and ergodicity of the %measure $\mu^{\lambda_1,...,\lambda_2}$ comes from existence, uniqueness, stationarity and ergodicity of $F^{(n)}$.
The proof of (\ref{Qd}) is by induction on $n$ and as it is a bit technical, we first prove the theorem for the case where $n=2$. We then continue to prove the induction for general $n$.\\

As observed at \remref{in_thm_Q_i}, the result for $n=1$
is a well-known property of $M/M/1$ queues.

Fix $a\geq0$ and $b>0$. Define
an event $A_{\epsilon}$ as follows: the process $F^{\left(2\right)}$
contains $a$ $1$\textquoteright s followed by $b$ $2$\textquoteright s
within the time interval $\left(0,\epsilon\right)$. As $\epsilon$
gets small this event becomes unlikely; we will look at the dominant
contribution to the probability computed in two different ways.

Firstly, by definition of $F^{(2)}$ as a fixed point,
$F^{\left(2\right)}$ is the output process of a rate-$1$ server
with arrival process also distributed as $F^{\left(2\right)}$, and
hence with queue distributed as $\mathbb{P}^{\left(1\right)}_{\lambda_1,\lambda_2}$. If $\epsilon$
is very small, the dominant way to get the event $A_{\epsilon}$ is
not to rely on any arrivals to the queue, but to suppose that the
queue already contains precisely $a$ $1$\textquoteright s and at
least $b$ $2$\textquoteright s at time $0$, and then that we see
$a+b$ services before time $\epsilon$. The probability of this event will
decay as $\epsilon^{a+b}$ and any other way of achieving it decays
quicker. Since the rate of service is $1$, we get
\begin{equation}
\mathbb{P}\left(A_{\epsilon}\right)\sim\mathbb{P}^{\left(1\right)}_{\lambda_{1},\lambda_{2}}\left(Q_{1}=a,Q_{2}\geq b\right)\frac{\epsilon^{\left(a+b\right)}}{\left(a+b\right)!},\label{eq:marginals_multi_type}
\end{equation}
where by $f(\epsilon) \sim g(\epsilon)$ we mean that $\nicefrac{f(\epsilon)}{g(\epsilon)}\rightarrow 1$ as $\epsilon \rightarrow 0$.

Alternatively, by Proposition \ref{prop:ntn}, $F^{\left(2\right)}$ is the output process of rate-$(\lambda_1+\lambda_2)$ server fed by $F^{(1)}$ (which is just a
Poisson process of rate $\lambda_{1}$), with unused services designated
as type-$2$ departures. In terms of such a queue, the dominant way
to get the event $A_{\epsilon}$ as $\epsilon\rightarrow0$ is for
the queue to contain precisely $a$ $1$\textquoteright s at time
$0$, and then to see $a+b$ services before time $\epsilon$. Again
this is better than relying on any new arrivals to the queue. In this case we get
\begin{equation}
\mathbb{P}\left(A_{\epsilon}\right)\sim\mathbb{P}^{\left(\lambda_{1}+\lambda_{2}\right)}_{\lambda_{1}}\left(Q_{1}=a\right)\frac{\left(\epsilon\left(\lambda_{1}+\lambda_{2}\right)\right)^{\left(a+b\right)}}{\left(a+b\right)!}\label{eq:marginals_multi_type-1}.
\end{equation}
Comparing (\ref{eq:marginals_multi_type}) and (\ref{eq:marginals_multi_type-1})
we get
\begin{align*}
\mathbb{P}^{\left(1\right)}\left(Q_{1}=a,Q_{2}\geq b\right) & =\left(\lambda_{1}+\lambda_{2}\right)^{\left(a+b\right)}\mathbb{P}^{\left(\lambda_{1}+\lambda_{2}\right)}_{\lambda_{1}}\left(Q_{1}=a\right)\\
&
=\left(\lambda_{1}+\lambda_{2}\right)^{\left(a+b\right)}\mathbb{P}^{\left(1\right)}_{\nicefrac{\lambda_{1}}{\lambda_{1}+\lambda_{2}},\nicefrac{\lambda_{2}}{\lambda_{1}+\lambda_{2}}}\left(Q_{1}=a\right)\\
&
 =\left(\lambda_{1}+\lambda_{2}\right)^{\left(a+b\right)}\left(1-\frac{\lambda_{1}}{\lambda_{1}+\lambda_{2}}\right)\left(\frac{\lambda_{1}}{\lambda_{1}+\lambda_{2}}\right)^{a}\\
 & =\left(1-\lambda_{1}\right)\lambda_{1}^{a}\frac{\lambda_{2}}{1-\lambda_{1}}\left(\lambda_{1}+\lambda_{2}\right)^{b-1}\\
 & =\mathbb{P}^{\left(1\right)}\left(Q_{1}=a\right)\frac{\lambda_{2}}{1-\lambda_{1}}\left(\lambda_{1}+\lambda_{2}\right)^{b-1},
\end{align*}
where in the second equality we used Remark \ref{rem}. From this it follows quickly that under $\mathbb{P}^{\left(1\right)}_{\lambda_{1},\lambda_{2}}$, $Q_{1}$
and $Q_{2}$ are independent, and $Q_{2}$ has Bernoulli-geometric
distribution with parameters $\nicefrac{\lambda_{2}}{(1\text{\textminus}\lambda_{1})}$
and $\lambda_{1}+\lambda_{2}$ as claimed.

We now turn to the proof for general $n\in \mathbb{Z}$. As an initial part of the
induction step for general $n$, it is useful to state
a lemma relating $\mathbb{P}^{(\lambda_{1}+...+\lambda_{n})}_{\lambda_{1},...,\lambda_{n-1}}$
and $\mathbb{P}^{(1)}_{\lambda_{1},...,\lambda_n}$.
\begin{lem}
Assume the induction hypothesis (\ref{Qd}) for $n-1$. Then for all $a_{1},...,a_{n-1}\in\mathbb{Z}_{+}$,
%\begin{align}
%\left(\lambda_{1}+...+\lambda_{n}\right)^{\sum_{j=1}^{n-1}a_{j}}\mathbb{P}^{\left(\lambda_{1}+...+\lambda_{n}\right)}\left(Q_{1}=a_{1},...,Q_{n-1}=a_{n-1}\right)\label{eq:marginals_multi_type-2}\\
%=\frac{\lambda_{n}}{\lambda_{1}+...+\lambda_{n}}\frac{1}{1-\left(\lambda_{1}+...+\lambda_{n-1}\right)}\nonumber \\
%\times\mathbb{P}^{\left(1\right)}\left(Q_{1}=a_{1},...,Q_{n-1}=a_{n-1}\right).\nonumber
%\end{align}
\begin{multline}
\label{eq:marginals_multi_type-2}
\left(\lambda_{1}+...+\lambda_{n}\right)^{\sum_{j=1}^{n-1}a_{j}}\mathbb{P}^{\left(\lambda_{1}+...+\lambda_{n}\right)}\left(Q_{1}=a_{1},...,Q_{n-1}=a_{n-1}\right)
\\
=\frac{\lambda_{n}}{\lambda_{1}+...+\lambda_{n}}\frac{1}{1-\left(\lambda_{1}+...+\lambda_{n-1}\right)}
\mathbb{P}^{\left(1\right)}\left(Q_{1}=a_{1},...,Q_{n-1}=a_{n-1}\right).
\end{multline}
\end{lem}
\begin{proof}
Recall that we can move from $\mathbb{P}^{(1)}$ to $\mathbb{P}^{(\lambda_{1}+\cdot\cdot\cdot+\lambda_{n})}$
by replacing $\lambda_{i}$ by $\nicefrac{\lambda_i}{(\lambda_{1}+\dots+\lambda_{n})}$
for each $i$. By the induction hypothesis, under $\mathbb{P}^{(1)}$
, the $Q_{i}$ , $1\leq i\leq n-1$ are independent with distribution
given by (\ref{Qd}). Hence they are also independent
under $\mathbb{P}^{(\lambda_{1}+\cdot\cdot\cdot+\lambda_{n})}$ .
It will be enough to show that for each $i$, for any $a_{i}\in\mathbb{Z}$,
\begin{equation}
\left(\lambda_{1}+...+\lambda_{n}\right)^{a_{i}}\mathbb{P}^{\left(\lambda_{1}+...+\lambda_{n}\right)}\left(Q_{i}=a_{i}\right)\label{eq:marginals_multi_type-3}
=\frac{\lambda_{i+1}+...+\lambda_{n}}{\lambda_{i}+...+\lambda_{n}}\frac{1-\left(\lambda_{1}+...+\lambda_{i-1}\right)}{1-\left(\lambda_{1}+...+\lambda_{i}\right)}\mathbb{P}^{\left(1\right)}\left(Q_{i}=a_{i}\right).
\end{equation}
Then the claim in (\ref{eq:marginals_multi_type-2}) will follow by
multiplying (\ref{eq:marginals_multi_type-3}) over $i=1,2,...,n-1$
and telescoping the products. The distribution of $Q_{i}$ under $\mathbb{P}^{(1)}$
is given by (\ref{Qd}), so to obtain the distribution
of $Q_{i}$ under $\mathbb{P}^{(\lambda_{1}+\cdot\cdot\cdot+\lambda_{n})}$,
we rescale the parameters as above to get
\[
Q_{i}\sim\text{Ber}\left(\frac{\lambda_{i}}{\lambda_{i}+...+\lambda_{n}}\right)\text{Geom}\left(\frac{\lambda_{1}+...+\lambda_{i}}{\lambda_{1}+...+\lambda_{n}}\right)
\]
Now one can check (\ref{eq:marginals_multi_type-3}) directly for
each value of $a_{i}$. There are essentially two cases, $a_{i}=0$
and $a_{i}>0$ (corresponding to the two forms of the probability
for a Bernoulli-geometric random variable).
\end{proof}
Now we can carry out the induction step. Following what we did in
the case $n=2$, fix $a_{1},...,a_{n-1}\geq0$ and $b\geq1$. Let
$A_{\epsilon}$ be the event that, during the time interval $\left(0,\epsilon\right)$,
the process $F^{(n)}$ contains $a_{1}$ $1$\textquoteright s, then
$a_{2}$ $2$\textquoteright s, and so on up to $a_{n-1}$ points
of type $(n-1)$, and then finally $b$ points of type $n$. Again
we let $\epsilon\rightarrow0$ and look at two ways of approximating
the probability of the event $A_\epsilon$. First we look at $F^{(n)}$ as the output of a queue
of rate $1$ fed by an arrival process whose distribution is $F^{(n)}$. As $\epsilon$ becomes small, the dominant way for the event $A_{\epsilon}$
to occur is that at time $0$ the queue already contains precisely
$a_{i}$ customers of type $i$ for $1\leq i\leq n-1$, and at least
$b$ customers of type $n$, and that then $a_{1}+\cdot\cdot\cdot+a_{n-1}+b$
services occur during the interval $(0,\epsilon)$. This gives
\begin{equation}
\mathbb{P}\left(A_{\epsilon}\right)\sim\mathbb{P}^{\left(1\right)}\left(Q_{1}=a_{1},...,Q_{n-1}=a_{n-1},Q_{n}\geq b\right)\frac{\epsilon^{\left(b+\sum_{j=1}^{n-1}a_{i}\right)}}{\left(b+\sum_{j=1}^{n-1}a_{i}\right)!}.\label{eq:marginals_multi_type-4}
\end{equation}
On the other hand, look at $F^{(n)}$ as the output of a queue of
rate $\lambda_{1}+\cdot\cdot\cdot+\lambda_{n}$, fed by an arrival
process whose distribution is $F^{(n-1)}$, and with points of type
$n$ added at times of unused service. Then the dominant way for $A_{\epsilon}$
to occur for small $\epsilon$ is that at time 0 the queue already
contains precisely $a_{i}$ customers of type $i$ for $1\leq i\leq n-1$,
and then $a_{1}+...+a_{n-1}+b$ services occur during $(0,\epsilon)$.
This leads to
\begin{equation}
\mathbb{P}\left(A_{\epsilon}\right)\sim\mathbb{P}^{\left(\lambda_{1}+...+\lambda_{n}\right)}\left(Q_{1}=a_{1},...,Q_{n-1}=a_{n-1}\right)\frac{\left(\epsilon\left(\lambda_{1}+...+\lambda_{n}\right)\right)^{b+\sum_{j=1}^{n-1}a_{i}}}{\left(b+\sum_{j=1}^{n-1}a_{i}\right)!}.\label{eq:marginals_multi_type-5}
\end{equation}
Comparing (\ref{eq:marginals_multi_type-4}) and (\ref{eq:marginals_multi_type-5})
and continuing using \eqref{marginals_multi_type-2}, we get that
for $b\geq1$,
\begin{align}
&\mathbb{P}^{\left(1\right)}\left(Q_{1}=a_{1},...,Q_{n-1}=a_{n-1},Q_{n}\geq b\right)\nonumber \\
&=\left(\lambda_{1}+...+\lambda_{n}\right)^{\left(b+\sum_{j=1}^{n-1}a_{i}\right)}\mathbb{P}^{\left(\lambda_{1}+...+\lambda_{n}\right)}\left(Q_{1}=a_{1},...,Q_{n-1}=a_{n-1}\right)\nonumber \\
&=\frac{\lambda_{n}}{\lambda_{1}+...+\lambda_{n}}\frac{1}{1-\left(\lambda_{1}+...+\lambda_{n-1}\right)}\mathbb{P}^{\left(1\right)}\left(Q_{1}=a_{1},...,Q_{n-1}=a_{n-1}\right)\left(\lambda_{1}+...+\lambda_{n}\right)^{b}\nonumber \\
&=\mathbb{P}^{\left(1\right)}\left(Q_{1}=a_{1},...,Q_{n-1}=a_{n-1}\right)\frac{\lambda_{n}}{1-\left(\lambda_{1}+...+\lambda_{n-1}\right)}\left(\lambda_{1}+...+\lambda_{n}\right)^{\left(b-1\right)}.\label{eq:marginals_multi_type-6}
\end{align}
From \eqref{marginals_multi_type-6} we see that $Q_{n}$ is independent
of $Q_{1},...,Q_{n-1}$ , and has the Bernoulli- geometric distribution
of (\ref{Qd}) as claimed. This completes the induction
step and the proof .
\end{proof}

\subsection{Two-column distribution in stationarity}\label{sec:TCS}
Using the same ideas as in the preceding proof we can also say something
about two neighbouring queues. Denote by $Q_{i}^{z}$ the number of
particles of type $i$ in the column $z$ in equilibrium and define $\mathbb{P}^{2,\left(1\right)}$ as the joint probability of two
queues $\mathbb{Q}^{0},\mathbb{Q}^{1}$(the joint distribution of
$\left(Q_{1}^{0},...,Q_{n}^{0},Q_{1}^{1},...,Q_{n}^{1}\right)$) ,
where the departure of $\mathbb{Q}^{0}$ is the arrival process of
$\mathbb{Q}^{1}$ and where the server process of both queues is of
rate $1$.
\begin{lem}
\label{lem:two_column_distribution}Let $\mathbb{P}^{2,\left(1\right)}$
be the distribution of two queues in tandem at stationarity, then
\[
\mathbb{P}^{2,\left(1\right)}\left(Q_{1}^{0}\geq1,Q_{1}^{1}=a,Q_{2}^{1}\geq b\right)=\lambda_{2}\lambda_{1}^{a+1}\left(\lambda_{1}+\lambda_{2}\right)^{b}.
\]
\end{lem}
\begin{proof}
We think of the process $F^{\left(2\right)}$ in two ways. First we
think of it as the departure process of the queue $\mathbb{Q}^{1}$
fed by $F^{\left(2\right)}$ and served at rate $1$. Let $N^1$ and $N^0$ two independent Poisson processes of rate $1$ that are independent of $\mathbb{Q}^0$ and $\mathbb{Q}^1$. Let $A$ be the event where one
sees in the departure process the sequence that begins with $a$ 1's then $b$ 2's and then
one $1$ in the time interval $[0,\epsilon)$. The probability of $A$ is dominated by
\begin{align}\label{eq:two_column_distribution}
&\mathbb{P}^{2,\left(1\right)}\left(Q_{1}^{0}\geq1,Q_{1}^{1}=a,Q_{2}^{1}\geq b\right)\mathbb{P}(\text{In the interval $[0,\epsilon]$, $N^1$ has exactly $a+b$ epochs}\\ &\text{   before $N^0$ has its first epoch  after which there is another epoch of $N^1$}).\nonumber
\end{align}
To see that, note that we need to have at least one first class particle in $\mathbb{Q}^{0}$,
$a$ first class particles in $\mathbb{Q}^{1}$ and at least $b$
second-class particles in $\mathbb{Q}^{1}$, then, in the time interval
$[0,\epsilon)$ the following must happen in order:
\begin{enumerate}
	\item $\left(a+b\right)$ customers must be served in $\mathbb{Q}^{1}$ before any customer is served in $\mathbb{Q}^{0}$;
	\item one service in $\mathbb{Q}^{0}$;
	\item one service in $\mathbb{Q}^{1}$.
\end{enumerate}
Recall that if $X$ is the sum of $n$ i.i.d. exponential r.v's of rate $\lambda$ then $X\sim \Gamma(n,\lambda)$, i.e.
\begin{align}
	\mathbb{P}(X\in dx)=f^{n,\lambda}(x)dx=\frac{\lambda^n x^{n-1}}{(n-1)!}e^{-\lambda x}dx.
\end{align}
 We have
\begin{align}\label{tcd1}
\mathbb{P}(A)&\sim\mathbb{P}^{2,\left(1\right)}\left(Q_{1}^{0}\geq 1,Q_{1}^{1}=a,Q_{2}^{1}\geq b\right)\int_{0}^{\epsilon}f^{a+b,1}(r_1)\int_{r_1}^{\epsilon}f^{1,1}(r_2)(1-e^{-(\epsilon-r_2)})dr_2dr_1 \\\nonumber
&\sim\mathbb{P}^{2,\left(1\right)}\left(Q_{1}^{0}\geq 1,Q_{1}^{1}=a,Q_{2}^{1}\geq b\right)\int_{0}^{\epsilon}\Big[\frac{r_1^{a+b-1}}{(a+b-1)!}e^{-r_1}\Big]\int_{r_1}^{\epsilon}e^{-(r_2-r_1)}(1-e^{-(\epsilon-r_2)})dr_2dr_1\\
&\sim\mathbb{P}^{2,\left(1\right)}\left(Q_{1}^{0}\geq 1,Q_{1}^{1}=a,Q_{2}^{1}\geq b\right)\int_{0}^{\epsilon}\Big[\frac{r_1^{a+b-1}}{(a+b-1)!}\Big]\int_{r_1}^{\epsilon}e^{-r_2}(1-e^{-(\epsilon-r_2)})dr_2dr_1\nonumber\\
&\sim\mathbb{P}^{2,\left(1\right)}\left(Q_{1}^{0}\geq1,Q_{1}^{1}=a,Q_{2}^{1}\geq b\right)\int_{0}^{\epsilon}\frac{r_1^{a+b-1}}{(a+b-1)!}(\epsilon-r_1)^2dr_1.\nonumber
\end{align}
Using integration by parts twice
\begin{align}\label{ip}
	\int_{0}^{\epsilon}\frac{r_1^{a+b-1}}{(a+b-1)!}(\epsilon-r_1)^2dr_1&=2\int_{0}^{\epsilon}\frac{r_1^{a+b+1}}{(a+b+1)!}dr_1\\
	&=\frac{2\epsilon^{a+b+2}}{(a+b+2)!}\nonumber.
\end{align}
Plugging (\ref{ip}) into (\ref{tcd1})
\begin{align}
	\mathbb{P}(A)\sim\mathbb{P}^{2,\left(1\right)}\left(Q_{1}^{0}\geq1,Q_{1}^{1}=a,Q_{2}^{1}\geq b\right)\frac{2\epsilon^{a+b+2}}{(a+b+2)!}.
\end{align}
On the other hand, we can think of the two queues under $\mathbb{P}^{2,\left(\lambda_{1}+\lambda_{2}\right)}$,
that is, having $F^{\left(1\right)}$ as their arrival
process and served at rate $\lambda_{1}+\lambda_{2}$. We can obtain
$F^{\left(2\right)}$ by interpreting an unserved epoch in $\mathbb{Q}^{1}$
as a second-class particle. We thus have
\begin{align}\label{eq:two_column_distribution-2}
&\mathbb{P}(A)\sim\mathbb{P}^{2,\left(\lambda_{1}+\lambda_{2}\right)}\left(Q_{1}^{0}\geq1,Q_{1}^{1}=a\right)\mathbb{P}\Big(\text{$N^1$ has exactly $a+b$ epochs}\\ &\text{in the interval $[0,\epsilon]$ before $N^0$ has its first epoch in the interval $[0,\epsilon]$ after which $N^1$ has an epoch}\Big)\nonumber\\\nonumber
&\sim\mathbb{P}^{2,\left(\lambda_{1}+\lambda_{2}\right)}\left(Q_{1}^{0}\geq1,Q_{1}^{1}=a\right)\int_{0}^{\epsilon}f^{a+b,\lambda_1+\lambda_2}(r_1)\int_{r_1}^{\epsilon}f^{1,\lambda_1+\lambda_2}(r_2)\Big(1-e^{-(\lambda_1+\lambda_2)(\epsilon-r_2)}\Big)dr_1dr_2\\\nonumber
&\sim\mathbb{P}^{2,\left(\lambda_{1}+\lambda_{2}\right)}\left(Q_{1}^{0}\geq1,Q_{1}^{1}=a\right)\int_{0}^{\epsilon}\Big[\frac{(\lambda_1+\lambda_2)^{a+b}r_1^{a+b-1}}{(a+b-1)!}\Big]\int_{r_1}^{\epsilon}e^{-(\lambda_1+\lambda_2)r_2}(1-e^{-(\lambda_1+\lambda_2)(\epsilon-r_2)})dr_2dr_1\\
&\sim\mathbb{P}^{2,\left(\lambda_{1}+\lambda_{2}\right)}\left(Q_{1}^{0}\geq1,Q_{1}^{1}=a\right)(\lambda_1+\lambda_2)^{a+b+1}\int_{0}^{\epsilon}\frac{r_1^{a+b-1}}{(a+b-1)!}(\epsilon-r_1)^2dr_1\nonumber\\
&\sim\mathbb{P}^{2,\left(\lambda_{1}+\lambda_{2}\right)}\left(Q_{1}^{0}\geq1,Q_{1}^{1}=a\right)(\lambda_1+\lambda_2)^{a+b+1}\frac{2\epsilon^{a+b+2}}{(a+b+2)!}\nonumber.\nonumber
\end{align}
Comparing (\ref{tcd1}) and (\ref{eq:two_column_distribution-2}) and letting $\epsilon$ go to zero,
we conclude that
\begin{align*}
\mathbb{P}^{2,\left(1\right)}\left(Q_{1}^{0}\geq1,Q_{1}^{1}=a,Q_{2}^{1}\geq b\right) & =\mathbb{P}^{2,\left(\lambda_{1}+\lambda_{2}\right)}\left(Q_{1}^{0}\geq1,Q_{1}^{1}=a\right)\left(\lambda_{1}+\lambda_{2}\right)^{a+b+1}\\
 & =\mathbb{P}^{2,\left(\lambda_{1}+\lambda_{2}\right)}\left(Q_{1}^{0}\geq1\right)\mathbb{P}^{2,\left(\lambda_{1}+\lambda_{2}\right)}\left(Q_{1}^{1}=a\right)\left(\lambda_{1}+\lambda_{2}\right)^{a+b+1}\\
 & =\frac{\lambda_{1}}{\lambda_{1}+\lambda_{2}}\left(1-\frac{\lambda_{1}}{\lambda_{1}+\lambda_{2}}\right)\left(\frac{\lambda_{1}}{\lambda_{1}+\lambda_{2}}\right)^{a}\left(\lambda_{1}+\lambda_{2}\right)^{a+b+1}\\
 & =\lambda_{2}\lambda_{1}^{a+1}\left(\lambda_{1}+\lambda_{2}\right)^{b-1},
\end{align*}
where in the second equality we used the independence of the number of first class particles across different columns and in the third equality the fact that the distribution of $Q^i_1$ is geometric and Remark \ref{rem}.
\end{proof}
%\begin{thm}
%Let $\left\{ \mathbb{Q}_{i}\right\} _{i\in\mathbb{Z}}$ be a sequence
%of queues in tandem with arrival process $F^{\left(2\right)}$. Then
%the r.vs $\left\{ Q_{0}^{1},Q_{0}^{2},Q_{-1}^{1},Q_{-2}^{1}...,\right\} $
%are independent of $\left\{ Q_{1}^{1},Q_{1}^{2},Q_{2}^{1},Q_{2}^{2},...\right\} $
%\end{thm}

\section{marginals of the TAZRP speed process}\label{ms}

In this section we apply the results in Section \ref{sec:sq} to obtain more refined results on the speed process. We divide the results into two subsections. The first deals with the distribution of one column of the speed process, whereas the second deals with the distribution of two columns.

\subsection{Distribution of the speeds at a single column}

One may think of a column in the multi-type TAZRP in stationarity
as a queue with a countable number of classes. For example, the column
of the speed process, $U_{0}$, can be thought of as a marked point
process $\mathcal{P}$ on $[0,1]$. Each realization of $U_{0}$ is
a countable set of  numbers in $[0,1]$, where we label each number
in that set (the speed $U_{z,i}$ for some $i$ which is to be thought
of as the class of a particle) with a number in $\mathbb{N}$ that
denotes the number of particles in that class. For example, if $U_{0,i-1}>U_{0,i}=U_{0,i+1}=\nicefrac{1}{2}>U_{0,i+2}$
, then $\left(\frac{1}{2},2\right)\in\mathcal{P}$ , that is, there
are two particles of class $\nicefrac{1}{2}$ in the column. In what follows we would like to show that  the TAZRP speed process can be viewed as the continuum version of the stationary measure discussed in Subsection \ref{sec:TCS}. In fact, we prove this by approximating $U_{0}$ by $n$-type TAZRP for large $n$. \\
For each $n\in\mathbb{N}$, fix $1>x_{1}>...>x_{n+1}=0$,
and define the function $\fu_{\bf{x}}^{n}:[0,1]\rightarrow\mathcal{R}_{n}$
by
\begin{equation}\label{eq:increasing function}
\fu_{\bf{x}}^{n}\left(x\right)=-\min\left\{ i:x\geq x_{i}\right\} ,
\end{equation}
where $\textbf{x}=\left(x_{1},...,x_{n}\right)$. By Corollary \ref{col:SIM} applying the map $\fu_{\textbf{x}}^{n}$ on each element $\eta(z,i)$
of $\eta\in\mathcal{Z}$ gives an element of $\mathcal{Z}_{n}$
so that $\fu_{\textbf{x}}^{n}\left(\pi U\right)$ is a stationary
and ergodic distribution for the $n$-type TAZRP.

\begin{proof}[Proof of Theorem \ref{Poisson pic}]
	Let $\fu_{\textbf{x}}^{n}:[0,1]\rightarrow\mathcal{R}_{n}$ be the
	function defined in (\ref{eq:increasing function}) associated with
	 $x_{i}=1-in^{-1}$, for $1\leq i\leq n+1$. Applying $\fu_{\textbf{x}}^{n}$ on $\pi U_{0}$
%	$\pi U_{0}$ (as we are only interested in the distribution of one column, we leave out $\pi$),
we obtain an ergodic and stationary measure for the $n$-type TAZRP. By the uniqueness of the stationary and ergodic measures of the $n$-type TAZRP, we see that the queue $Q^n=(Q^n_1,...,Q^n_n)$ has a Bernoulli-geometric product distribution as in Theorem \ref{thm:mq}. The arrival rates to the queue $Q^n$ are given by
	\begin{align}
		\lambda_{i}= \sqrt{1-(i-1)n^{-1}}-\sqrt{1-in^{-1}}=\sqrt{x_{i-1}}-\sqrt{x_i}, \quad \text{for}\quad 1\leq i \leq n.
	\end{align}
	To see that, note that by stationarity of $\fu^{(n)}$, the arrival rate of customers of type $i$ to the queue equals the rate of departure of customers of type $i$ under $\mathbb{P}^{(1)}$. By the stationarity of $\mathbb{P}^{(1)}$, the rate of departure of particles of type $i$ equals the probability that the $i$'th customer is the first in the queue, that is, particle of type $i$ is next to be served in the queue.
	\begin{align}
		\lambda_i&=\mathbb{P}^{(1)}(Q_1=0,...,Q_{i-1}=0,Q_i>0)\\
		&=\mathbb{P}(U_{0,0}\in(x_{i-1},x_i])=1-\sqrt{x_{i}}-(1-\sqrt{x_{i-1}}),\nonumber
	\end{align}
	where in the third equality we used the marginal distribution of $U_{0,0}$ in (\ref{dou}). By Theorem \ref{thm:mq} we see that $Q^n_1,...,Q^n_n$ are independent and that
	\begin{align}\label{cq}
		&Q^n_{i}\sim\text{Ber}\left(\frac{\sqrt{x_{i-1}}-\sqrt{x_i}}{\sqrt{x_{i-1}}}\right)\text{Geom}\left(1-\sqrt{x_i}\right)\\
		&=\text{Ber}\left(\frac{\sqrt{1-(i-1)n^{-1}}-\sqrt{1-in^{-1}}}{\sqrt{1-(i-1)n^{-1}}}\right)\text{Geom}\left(1-\sqrt{1-in^{-1}}\right)\nonumber.
	\end{align}
	Fix $\epsilon>0$ and let $i\in [1,n-\lfloor \epsilon n\rfloor]$. In what follows we use $C$ to denote  a constant that may depend on some variables and that changes form line to line. Note that by Taylor expansion around $1-(i-1)n^{-1}$ there exists $C(i,\epsilon)>0$ such that
	\begin{align}\label{te}
		&\frac{\sqrt{1-(i-1)n^{-1}}-\sqrt{1-in^{-1}}}{\sqrt{1-(i-1)n^{-1}}}=\frac{(2\sqrt{1-(i-1)n^{-1}})^{-1}n^{-1}}{\sqrt{1-(i-1)n^{-1}}}+Cn^{-2}\\
		&=\frac{n^{-1}}{2(1-(i-1)n^{-1})}+Cn^{-2}=\frac{n^{-1}}{2x_{i-1}}+Cn^{-2},\nonumber
	\end{align}
	 where  for a fixed $\epsilon$ and every $n$, $C(\cdot,\epsilon)$ is bounded uniformly on $i\in [1,n-\lfloor \epsilon n\rfloor]$. Fix $y\in[\epsilon,1)$, and let $i_n=\lceil n y\rceil$. Then, $y\in [x_{n-i_n},x_{n-i_n-1})$, and if $y\in U_0$ then $Q^n_{n-i_n}\neq0$. Plugging  (\ref{te}) in (\ref{cq}), there exists a $C(i,\epsilon)$ such that
	\begin{align}\label{bg}
			Q^n_{n-i_n} \sim \text{Ber}\left(\frac{n^{-1}}{2x_{n-i_n-1}}+Cn^{-2}\right)\text{Geom}\left(1-\sqrt{x_{n-i_n-1}}-Cn^{-1}\right).
	\end{align}
	As $|y-x_{n-i_n-1}|\leq n^{-1}$, plugging $y$ into (\ref{bg}), there exists $C(y,\epsilon)>0$ where $C(\cdot,\epsilon)$ is bounded on $[\epsilon,1)$,  such that
	\begin{align}\label{q est}
		Q^n_{n-i_n} \sim \text{Ber}\left(\frac{n^{-1}}{2y}+Cn^{-2}\right)\text{Geom}\left(1-\sqrt{y}+Cn^{-1}\right).
	\end{align}
	 Now let $\mathcal{P}^n=\{(p_i,l_i)\}_{i=1}^n$ be the marked point process associated with the queue $Q^n$ by
	 \begin{align}
	 	p_i&=x_{n-i} \quad 1 \leq i\leq n\quad \text{(points)}\\
	 	l_i&=Q^n_{n-i} \quad 1 \leq i\leq n\quad \text{(marks)}.\nonumber
	 \end{align}
	 Let $\mathcal{P}^{n,1}=\bigcup_{l_i\neq 0}p_i$. If $t_1,t_2,t_3\in (\epsilon,1]$ such that $t_1<t_2<t_3$, then by Theorem \ref{thm:mq}
	 \begin{align}
	 	\#\{\mathcal{P}^{n,1}\cap [t_1,t_2)\}, \#\{\mathcal{P}^{n,1}\cap (t_2,t_3]\}
	 \end{align}
	 are independent. By (\ref{q est}) we see that
	 \begin{align}
	 	\delta^{-1}\lim_{\delta\rightarrow 0}\lim_{n\rightarrow\infty}\mathbb{P}(\mathcal{P}^{n,1}\cap [y,y+\delta)\neq \emptyset)=\frac{1}{2y},
	 \end{align}
	 which implies that $\mathcal{P}^{n,1}$ converges to an inhomogeneous Poisson process with intensity $\frac{1}{2x}$ on $[\epsilon,1)$. Next, again by (\ref{q est}), we see that conditioned on the event
	 \begin{align}
	 	\mathcal{P}^{n,1}\cap [y,y+\delta]\neq \emptyset,
	 \end{align}
	 if $p_i\in \mathcal{P}^{n,1}\cap [y,y+\delta)$, $l_i\sim \text{Geom}(r(\delta,n))$ and
	 \begin{align}
	 	\lim_{\delta\rightarrow 0}\lim_{n\rightarrow \infty}r(\delta,n)=1-\sqrt{y},
	 \end{align}
	 which implies the result on $[\epsilon,1)$. Taking $\epsilon\rightarrow 0$ concludes the proof.
\end{proof}
%\begin{cor}
%For a fixed $j>0$, the sequence of speeds $\left\{ U_{0,i}\right\} _{i=j+1}^{\infty}$
%conditioned on $U_{0,j}$ is independent of $\left\{ U_{0,i}\right\} _{i=0}^{j-1}$.
%\end{cor}
%\begin{proof}
%Condition on the event $A=\{U_{0,j}=v\}$, where $v\in[0,1]$. By Theorem \ref{Poisson pic} we know that $\cup_i U_{0,i}$ can be seen as a realization of a Poisson point process. This implies that condition on $A$
%\begin{align}
%	\bigcup_{l=0}^{j-1}U_{0,i}\text{ is independent of } \bigcup_{l=j+1}^{\infty}U_{0,i}.
%\end{align}
%This shows the independence of the \textbf{sets} of speeds below and above the speed $v$. That the labels of the speeds above $v$ is also independent of the labels of the speeds below $v$ is a consequence of the memoryless property of the geometric distribution.
%\end{proof}

\begin{proof}[Proof of Theorem \ref{thm:tdm}]
Fix $0=x_3<x_{2}<x_{1}<1$ and consider the map given in (\ref{eq:increasing function})
associated with $x_{1},x_{2}$. First note that $U_0$ and $\pi U_0$ have the same distribution and therefore, so do $\fu(\pi U)$ and $\fu(U)$ a fact we use in the computations below. Since $\fu\left(\pi U\right)$
is stationary w.r.t. the $2$- type TAZRP with some $\lambda_{1}$
and $\lambda_{2}$, we can relate $x_{1},x_{2}$ to $\lambda_{1},\lambda_{2}$. By the definition of the  projection $\fu(\pi U)$
\begin{align*}
\mathbb{P}\left(\left(\pi U\right)_{0,0}<x_{2}\right) & =\mathbb{P}\left( U_{0,0}<x_{2}\right)=\mathbb{P}^{\left(1\right)}\left(Q_{1}=0,Q_{2}=0\right)\\
\mathbb{P}\left(\left(\pi U\right)_{0,0}<x_{1}\right) & =\mathbb{P}\left( U_{0,0}<x_{1}\right)=\mathbb{P}^{\left(1\right)}\left(Q_{1}=0\right).
\end{align*}
Using Lemma \ref{lem:opm} and Theorem \ref{thm:mq}
we see that
\begin{align}\label{eq3}
\lambda_{1} & =1-\sqrt{x_{1}}\nonumber \\
\lambda_{2} & =\sqrt{x_{1}}-\sqrt{x_{2}},
\end{align}
and
\begin{align}\label{ec}
&\mathbb{P}^{(1)}(Q_1=k)=(1-\lambda_{1})\lambda_{1}^k\\
	&\mathbb{P}^{(1)}(Q_2=k)=\begin{cases}
	1-\frac{\lambda_{2}}{1-\lambda_{1}} & k=0\\\nonumber
	\frac{\lambda_{2}}{1-\lambda_{1}}(1-(\lambda_{1}+\lambda_{2}))(\lambda_{1}+\lambda_{2})^{k-1} & k>0.
	\end{cases}	
\end{align}
For $i<j$
\begin{align}\label{eq1}
\mathbb{P}\left(x_{1}>  U_{0,i} , x_{2}> U_{0,j}\right)& = \mathbb{P}^{\left(1\right)}\left(Q_{1}\leq i,Q_1+Q_{2}\leq j\right)\\
 & =\sum_{l=0}^{i}\mathbb{P}^{\left(1\right)}\left(Q_{1}= l\right)\mathbb{P}^{\left(1\right)}\left(Q_{2}\leq j-l\right).\nonumber
\end{align}
Using (\ref{ec})
\begin{align}\label{eq2}
\mathbb{P}^{(1)}\left(Q_{2}\leq m\right)&=\frac{1-(\lambda_{1}+\lambda_{2})}{1-\lambda_{1}}+\frac{\lambda_{2}(1-(\lambda_{1}+\lambda_{2})^{m})}{1-\lambda_{1}}\\
&=1-\frac{\lambda_{2}(\lambda_{1}+\lambda_{2})^m}{1-\lambda_{1}}.\nonumber
\end{align}
Plugging (\ref{eq2}) into (\ref{eq1}) and using (\ref{eq3})
\begin{align}
&\mathbb{P}\left(x_{1}\geq U_{0,i}, x_{2}\geq U_{0,j}\right)\\ &=\sum_{l=0}^{i}(1-\lambda_{1})\lambda_{1}^l\left(1-\frac{\lambda_{2}(\lambda_{1}+\lambda_{2})^{j-l}}{1-\lambda_{1}}\right)\nonumber\\
&=1-\lambda_{1}^{i+1}-\left(\lambda_{1}+\lambda_{2}\right)^{j+1}\left(1-\left(\frac{\lambda_{1}}{\lambda_{1}+\lambda_{2}}\right)^{i+1}\right)\nonumber\\
&=1-(1-\sqrt{x_1})^{i+1}-\left(1-\sqrt{x_2}\right)^{j+1}\left(1-\left(\frac{1-\sqrt{x_1}}{1-\sqrt{x_2}}\right)^{i+1}\right),\nonumber
\end{align}
which is what we wanted.\\
Next we compute the joint distribution on the diagonal, that is, the
probability that the $i$'th and $j$'th particles have the same speed.

\begin{align}\label{eq4}
 &\mathbb{P}\left(x_{1}>U_{0,i}\geq U_{0,j}>x_{2}\right) =\mathbb{P}^{\left(1\right)}\left(Q_{1}\leq i,Q_{1}+Q_{2}\geq j+1\right) \\
 & =\sum_{l=0}^{i}\mathbb{P}^{\left(1\right)}\left(Q_{1}=l\right)\mathbb{P}^{\left(1\right)}\left(Q_{2}\geq j+1-l|Q_{1}=l\right)\nonumber \\
 & =\sum_{l=0}^{i}\left[\left(1-\lambda_{1}\right)\lambda_{1}^{l}\right]\left[\frac{\lambda_{2}\left(\lambda_{1}+\lambda_{2}\right)^{j-l}}{1-\lambda_{1}}\right]\nonumber \\
 & =\lambda_{2}(\lambda_{1}+\lambda_{2})^j\sum_{l=0}^{i}\left(\frac{\lambda_{1}}{\lambda_{1}+\lambda_{2}}\right)^l\nonumber \\
 & =(\lambda_{1}+\lambda_{2})^{j+1}\left(1-\left(\frac{\lambda_{1}}{\lambda_{1}+\lambda_{2}}\right)^{i+1}\right),\nonumber
\end{align}
where we used the independence of $Q_{1}$ and $Q_{2}$. Plugging
(\ref{eq3}) into (\ref{eq4})
we obtain
\[
\mathbb{P}\left(x_{1}>U_{0,i}\geq U_{0,j}>x_{2}\right)=(1-\sqrt{x_2})^{j+1}\left(1-\left(\frac{1-\sqrt{x_1}}{1-\sqrt{x_2}}\right)^{i+1}\right).
\]
Dividing by $x_{1}-x_{2}$ and letting $x_{1}\rightarrow x_{2}$ we
conclude that
\[
\mathbb{P}\left(U_{0,i}=U_{0,j}\in dx\right)=(i+1)\frac{(1-\sqrt{x})^j}{2\sqrt{x}}dx.
\]
\end{proof}

At this point we can also give the proof of the uniqueness
statement in Theorem \ref{thm:sd}.
We wish to show that if $\mu$
is the distribution of the speed process,
so that $\mu^\pi$ is a stationary distribution of the
multi-type TAZRP, then every translation-invariant ergodic
stationary distribution of the multi-type TAZRP
is of the form $\fu(\mu^\pi)$ for some
non-decreasing function $\fu$.

\begin{proof}[Proof of the uniqueness statement in Theorem \ref{thm:sd}]
The coupling approach of Mountford and Prabhakar
\cite{MP95} shows that for given
$\lambda_1, \dots, \lambda_n$,
there is a unique translation-invariant ergodic
stationary distribution of the $n$-type TAZRP
$\mu_{\lambda_1,\dots,\lambda_n}$
such that the rate of jumps of particles of type $i$
from site $0$ to site $1$ is $\lambda_i$.
In fact, since in stationarity the rate of such jumps is just the probability that the highest-priority particle at site $0$
has type $i$, these distributions are characterised by
the distribution of the type of that particle; under
$\mu_{\lambda_1,\dots,\lambda_n}$, the probability that
the highest-priority particle at site $0$ has type $i$
is $\lambda_i$.

Any distribution $\nu$ on $\mathcal{Z}$ is characterised by
the probabilities of cylinder events of the form
\[
\left\{
\eta(z,1)\leq a_1, \dots, \eta(z,k)\leq a_k
\right\}.
\]
Hence in fact $\nu$ is characterised by its projections
$\fu^n_{\bf{x}}(\nu)$ where $\fu^n_{\bf{x}}$ is a function of the form defined
at (\ref{eq:increasing function}).

If $\nu$ is a stationary distribution for the multi-type TAZRP,
then we know that any such $\fu^n_{\bf{x}}(\nu)$ is stationary
for the $n$-type TAZRP. Suppose that $\nu$ and $\tilde{\nu}$
are two translation-invariant ergodic stationary distributions
for the multi-type TAZRP, such that the distribution
of $\eta(0,0)$ is the same under $\nu$ and $\tilde{\nu}$.
Then, by the characterisation of the
distributions $\mu_{\lambda_1,\dots,\lambda_n}$ above,
the $n$-type stationary distributions $\fu^n_{\bf{x}}(\nu)$ and $\fu^n_{\bf{x}}(\tilde{\nu})$ are in fact the same for any such
$\bf{x}$. Hence $\nu$ and $\tilde{\nu}$ are the same.
So for any given distribution of $\eta(0,0)$,
there is at most one translation-invariant ergodic stationary distribution.

But under $\mu^\pi$, the distribution of $\eta(0,0)$ is
non-atomic. So for any desired target distribution,
we can find a non-decreasing function $\fu$ with the desired
distribution of $\eta(0,0)$ under $\fu(\mu^\pi)$.
Hence indeed all translation-invariant ergodic stationary distributions are of the form $\fu(\mu^{\pi})$, as desired.
\end{proof}

%We are now ready for the proof of the uniqueness of $\mu^\pi$. The following proof shows that if $\nu$ is another distribution on $\mathcal{Z}$ that is ergodic, stationary with respect to the TAZRP dynamics and has column distribution $\mu^\pi_0$ then $\mu^\pi=\nu$.
%\begin{proof}[Proof of uniqueness in Theorem \ref{thm:sd}]
%Assume $\nu$ is a distribution on $\mathcal{Z}$ that is ergodic, stationary with respect to the TAZRP dynamics and has column distribution $\mu^\pi_0$. Let $F_{\bf{x}}^{n}$ be the function defined in  (\ref{eq:increasing function}). Let
%\begin{align}\label{uq}
%	\mu^{\pi,n}&=F_{\bf{x}}^n(\mu^\pi)\\
%	\nu^n&=F_{\bf{x}}^n(\nu)\nonumber.
%\end{align}	
%For each $n\geq 1$ the distributions $\mu^{\pi,n}$ and $\nu^n$ are $n$-type TAZRP ergodic and stationary distributions. As both distributions $\mu^{\pi}$ and $\nu$ have the same marginal column distribution $\mu^\pi_0$ so do the distributions $\mu^{\pi,n}$ and $\nu^n$ have the same marginal column distribution. By uniqueness of $n$-type distributions (Theorem \ref{thm:mq}) it follows that
%\begin{align}\label{c3}
%	\mu^{\pi,n}=\nu^n.
%\end{align}
%As in Theorem \ref{Poisson pic}
%\begin{align}
%	\lim_{n\rightarrow \infty} \mu^{\pi,n}&=\mu^\pi \label{c1}\\
%	\lim_{n\rightarrow \infty} \nu^n&=\nu.\label{c2}
%\end{align}
%(\ref{c3}) along with (\ref{c1})--(\ref{c2}) imply that $\mu^\pi=\nu$ and the result follows.
%\end{proof}
\subsection{Joint distribution of multiple columns}
In this section we apply the results in Section \ref{sec:TCS} to Proposition \ref{tc}.
\begin{proof}[Proof of Proposition \ref{tc}]
Let $\mathbb{Q}^{0}$ and $\mathbb{Q}^{1}$ be two queues with arrival
process $F^{\left(2\right)}$ in stationarity, s.t the departure process
of $\mathbb{Q}^{0}$ is the arrival process of $\mathbb{Q}^{1}$.
It is not hard to see that
\begin{align}
&\mathbb{P}\left(U_{0,0}>x_1,U_{-1,j-1}>x_{1}>U_{-1,j}>...>U_{-1,j+k-1}>x_{2}\right)\\
&=\mathbb{P}\left(\left(\pi U\right)_{0,0},\left(\pi U\right)_{1,j-1}>x_{1}>\left(\pi U\right)_{1,j}>...>\left(\pi U\right)_{1,j+k-1}>x_{2}\right)\nonumber\\&=\mathbb{P}^{\left(1\right)}\left(Q_{1}^{0}\geq1,Q_{1}^{1}=j,Q_{2}^{1}\geq k\right)\nonumber.
\end{align}
By \lemref{two_column_distribution} we see that
\[
\mathbb{P}^{\left(1\right)}\left(Q_{1}^{0}\geq1,Q_{1}^{1}=j,Q_{2}^{1}\geq k\right)=\lambda_{2}\lambda_{1}^{j}\left(\lambda_{1}+\lambda_{2}\right)^{k-1}.
\]
Using (\ref{eq3})
we obtain the result.
\end{proof}
\begin{rem}
	Fix $v\in(0,1)$. Apply the function $\fu_{v}^{1}$ on the reflected speed process $\pi U$ and define the events
	\begin{align}
	A_{i}=\left\{ \text{the number of particles in the column $\fu_v^1(\pi U)_{i,\cdot}$  whose speed exceeds \ensuremath{v}}\right\} \quad i\in \mathbb{Z}.
	\end{align}
	As the distribution of $\fu_v^1(\pi U)$ is stationary with respect to the $1$-type TAZRP we see that the events $\{A_i\}_{i\in \mathbb{Z}}$ are independent.
	
\end{rem}
\section{overtaking}\label{sec:ot}

%Here we intend to do the following:
%\begin{enumerate}
%\item Examples - layered (inverese lexicographical order, maybe add question
%about total orderings)
%\item Seeing the TASEP speed process in the arrival process - James
%\end{enumerate}
Consider the initial condition $\eta^*$. Let $i_{1}\leq i_{2}$ and $j_1,j_2$ be such that $p_{i_1,j_1}>p_{i_2,j_2}$. We define their meeting time $T\in\mathbb{R}_{+}\cup\infty$
as the first time that the particle $p_{i_{1},j_{1}}$ is at the same
column as $p_{i_{2}j_{2}}$. We say the particle $p_{i_{1},j_{1}}$
overtakes the particle $p_{i_{2},j_{2}}$ if $T<\infty$. Note that
\begin{align}	
	&X_{p_{i_{1},j_{1}}}\left(t\right)<X_{p_{i_{2},j_{2}}}\left(t\right) \quad t<T\nonumber\\
	&X_{p_{i_{1},j_{1}}}\left(t\right)\geq X_{p_{i_{2},j_{2}}}\left(t\right) \quad t\geq T.\nonumber
\end{align}

\begin{proof}[Proof of Theorem \ref{thm:ov}]
The case where $U_{0,j}>U_{i,k}$ is clear and so we assume $U_{0,j}=U_{i,k}$.
The proof will rely on the following observations:
\begin{enumerate}
\item \label{en:relative}
As we are concerned only with the positions of the particles $p_{0,j}$
and $p_{i,k}$, we may change the types of the particles in the configuration
(even using a {\it non-monotone} relabelling) as long as the relative priority is preserved {\it with respect to the
two particles}.
\item We may also ignore any part of the dynamics that does not affect the positions of the two particles.
\item  We are only interested in the dynamics until the overtaking time $T$.
\end{enumerate}
We will use these guidelines to simplify the TAZRP configuration. We will then use the coupling with the TASEP and results on the TASEP speed process to conclude that overtaking occurs.

We divide the proof into several cases according to the values of $i,j,k$.\\
{\bf Case 1:} First assume that $i=1$ and $j=k=0$, in other words,
we assume the particle $p_{0,0}$ is at the bottom of the column $0$
has the same speed as that of the particle at the bottom of column
$1$. As the particles above $p_{0,0}$ are weaker then $p_{0,0}$ we may consider them as holes with respect to both $p_{0,0}$ and $p_{1,0}$, as this would not affect the dynamics of the two particles until overtaking occurs.
We now re-label the rest of the particles as follows (Figure \ref{fig:points}):
\begin{itemize}
  \item $p_{0,0}=-2$ and $p_{1,0}=-3$.
  \item $p_{i,j}=-1$ for all $i<0$, $j\in\mathbb{N}_{0}$.
  \item $p_{i,j}=-4$ for $i>1$ and for $i=1, j>0$.
\end{itemize}
It is straightforward to check that this keeps the order of priority with respect to the two particles. Using the coupling of the TAZRP with
the TASEP we see that this configuration translates to (for the TASEP we use the convention in \cite{amir2011tasep} that stronger particles are those with smaller value)
\begin{equation}\label{eq:same_speed_overtaking}
\ldots1111234444\ldots
\end{equation}
Recall from Subsection \ref{ssec:scc} that the coupling of the TAZRP with several second-class particles
holds until the first time that two second-class particles are in
the same column, which in this case is up to time $T$. Since particle
$p_{0,0}$ and $p_{1,0}$ have the same speed, so do the particles $2$ and $3$ in (\ref{eq:same_speed_overtaking}). In \cite[Theorem 1.14]{amir2011tasep} it was shown that with probability $1$ particle $2$ overtakes particle $3$. By the coupling with the TAZRP we see that $p_{0,0}$ overtakes $p_{1,0}$. \\
{\bf Case 2:} Next assume $i=1, j\geq 0$, $k\geq0$. (if $j=k=0$ this degenerates back to the previous case). Here we label the particles in the same
way as before except for $p_{m,l}$ where $m=0$ and $0\leq l<j$
or $m=1$ and $0\leq l<k$ which are labeled as first class particles,
i.e. $p_{m,l}=-1.$ Although the latter ($m=1$ and $0\leq l<k$ )
are of smaller value than the particle $p_{0,j}$, until time $T$ there is no interaction between them and $p_{0,j}$ (as they are always strictly to the right of $p_{0,j}$) so the labelling is consistent with the dynamics up to the point of
overtaking. This translates to the following multi-type TASEP configuration
\begin{equation}\label{eq:same_speed_overtaking-1}
\ldots11112\underset{j}{\underbrace{1...1}}3\underset{k}{\underbrace{1...1}}4444\ldots
\end{equation}
We now claim that particle $2$ overtakes particle $3$ in (\ref{eq:same_speed_overtaking-1}).
Assume it does not, then there is some positive probability $p>0$
of reaching (\ref{eq:same_speed_overtaking-1}) from (\ref{eq:same_speed_overtaking}).
This implies that starting from (\ref{eq:same_speed_overtaking})
with some positive probability particle $2$ will not overtake particle
$3$ contradicting that starting from configuration (\ref{eq:same_speed_overtaking})
particle $2$ a.s. overtakes particle $3$. \\
{\bf Case 3} Finally we prove the theorem for $i>1$. We use induction on $i$.
Suppose $U_{0,j}=U_{i+1,k}$ for some $j,k\in\mathbb{N}_{0}$, and
that our hypothesis holds for $1\leq i'\leq i$. There are two possibilities:
\begin{enumerate}%[label=(\alph*)]
  \item \label{it:same} There exists $1\leq m\leq i$ and $l\in\mathbb{N}_{0}$ s.t $U_{0,j}=U_{m,l}=U_{i+1,k}$.
  \item \label{it:different} For every $1\leq m\leq i$ and $l\in\mathbb{N}_{0}$ $U_{m,l}\neq U_{0,j}=U_{i+1,k}$.
\end{enumerate}
For case \ref{it:same} we use the induction hypothesis twice to conclude that particle $p_{m.l}$
overtakes particle $p_{i+1,k}$ and that particle $p_{0,j}$ overtakes
particle $p_{m,l}$ which together implies that $p_{0,j}$ overtakes $p_{i+1,k}.$
It remains therefore to deal with case \ref{it:different}.
Note that by (\ref{eq:one_point_marginal_TAZRP-2}) we see that for
every $m\in\mathbb{Z}$ w.p $1$ we have
\begin{equation}\label{eq:same_speed_overtaking-2}
\lim_{l\rightarrow\infty}U_{m,l}=0.
\end{equation}
Equation (\ref{eq:same_speed_overtaking-2}) implies that for $1\leq m\leq i$,
column $m$ has only a finite number of particles whose speed exceeds
$U_{0,j}$ and all the speeds of all other particles in the column
are strictly smaller than $U_{0,j}$. Particles located at column
$m$ for $1\leq m\leq i$ and whose speed is smaller than $U_{i+1,k}$
cannot overtake particle $p_{i+1,k}$ and are of value smaller than
that of $p_{0,j}$ and therefore will not change the dynamics of $p_{0,j}$
and $p_{i+1,k}$ and can be considered as holes by both particles. Particles located at column $m$
for $1\leq m\leq i$ and whose speed is larger than $U_{i+1,k}$ cannot
be overtaken by particle $p_{0,j}$ (whose speed is smaller then theirs
but whose value is greater) and we can therefore label them as first
class particles. Since we are in case \ref{it:different} all particles in columns $1\leq m \leq i$ fall into one of the above two options. The particles at columns $0$ and $i+1$ are labelled as before. All together, we see that we can translate our TAZRP
configuration to the following TASEP configuration
\begin{equation}
\ldots11112\underset{j}{\underbrace{1...1}}\underset{N_{1}}{\underbrace{1...1}}4\underset{N_{i}}{\underbrace{1...1}}3\underset{k}{\underbrace{1...1}}4444\ldots...,\label{eq:same_speed_overtaking-3}
\end{equation}
where for $1\leq m\leq i$ $N_{m}$ is the number of particles whose
speed is greater than $U_{0,j}$. We now argue as before: assume particle
$2$ does not overtake particle $3$, then there is some positive
probability $p'>0$ of reaching (\ref{eq:same_speed_overtaking-3})
from (\ref{eq:same_speed_overtaking}). This implies that starting
from (\ref{eq:same_speed_overtaking}), with some positive probability
particle $2$ will not overtake particle $3$ contradicting that starting
from configuration (\ref{eq:same_speed_overtaking}) particle $2$
a.s. overtakes particle $3$. We have now proved the inductive step.
As we already proved the hypothesis of the induction for $i=0$ (particles
$2$ and $3$ are located in two adjacent columns) the result follows.
\end{proof}
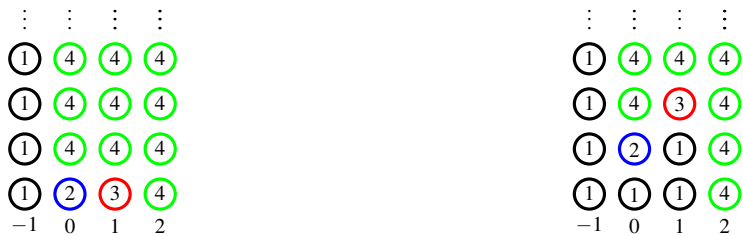
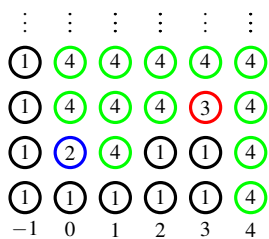
\begin{figure}[ht!]
	\centering%
	\begin{subfigure}[t]{.4\textwidth}
		\centering
		\begin{tikzpicture}[scale=0.4, every node/.style={transform shape}]\label{case 1}
	%\draw [line width=0.01cm] (0,0) -- (8,0);
%	\foreach \x in {1,...,4}
%	{   \draw  [very thick](0.5,0.05+\x*1.5) circle [radius=0.5][];
%		\node [scale=2][above] at (0.5,0.1+\x*1.5-0.5) {$1$};
%		\draw [dashed][line width=0.01cm] (0.5,7) -- (0.5,8);
%	};
	\foreach \x in {-1,...,2}
	{
		\node [scale=2][above] at (3.5+\x*1.5,0) {$\x$};
	};
	\node [scale=2][above] at (2,6.8) {$\vdots$};
	\foreach \x in {1,...,4}
	{   \draw  [very thick](2,0.05+\x*1.5) circle [radius=0.5][];
		\node [scale=2][above] at (2,0.1+\x*1.5-0.5) {$1$};
	};
		\foreach \x in {2,...,4}
	{   \draw  [very thick](3.5,0.05+\x*1.5) circle [radius=0.5][green];
		\node [scale=2][above] at (3.5,0.1+\x*1.5-0.5) {$4$};
	};
	\node [scale=2][above] at (3.5,6.8) {$\vdots$};
		\foreach \x in {2,...,4}
	{   \draw  [very thick](5,0.05+\x*1.5) circle [radius=0.5][green];
		\node [scale=2][above] at (5,0.1+\x*1.5-0.5) {$4$};
	\node [scale=2][above] at (5,6.8) {$\vdots$};
	};
	\foreach \x in {1,...,4}
{   \draw  [very thick](6.5,0.05+\x*1.5) circle [radius=0.5][green];
	\node [scale=2][above] at (6.5,0.1+\x*1.5-0.5) {$4$};
	\node [scale=2][above] at (6.5,6.8) {$\vdots$};
};
	\draw  [very thick](3.5,0.05+1.5) circle [radius=0.5][blue];
	\node [scale=2][above] at (3.5,0.1+1.5-0.5) {$2$};
	\draw  [very thick](5,0.05+1.5) circle [radius=0.5][red];
	\node [scale=2][above] at (5,0.1+1.5-0.5) {$3$};
		\end{tikzpicture}
		\caption{\small Case 1. The particles above particles $2$ and $3$ are of lower class and do not affect the dynamics up to the meeting time.}
	\end{subfigure}%
\hspace{2em}
	\begin{subfigure}[t]{{.5\textwidth}}
		\centering
		\begin{tikzpicture}[scale=0.4, every node/.style={transform shape}]
%		%column 0
%	\foreach \x in {1,...,4}
%{   \draw  [very thick](0.5,0.05+\x*1.5) circle [radius=0.5][];
%	\node [scale=2][above] at (0.5,0.1+\x*1.5-0.5) {$1$};
%	\draw [dashed][line width=0.01cm] (0.5,7) -- (0.5,8);
%};
%column 1
	\foreach \x in {-1,...,2}
{
	\node [scale=2][above] at (3.5+\x*1.5,0) {$\x$};
};
\node [scale=2][above] at (2,6.8) {$\vdots$};
\foreach \x in {1,...,4}
{   \draw  [very thick](2,0.05+\x*1.5) circle [radius=0.5][];
	\node [scale=2][above] at (2,0.1+\x*1.5-0.5) {$1$};
};
%column 2
\foreach \x in {3,...,4}
{   \draw  [very thick](3.5,0.05+\x*1.5) circle [radius=0.5][green];
	\node [scale=2][above] at (3.5,0.1+\x*1.5-0.5) {$4$};
	\node [scale=2][above] at (3.5,6.8) {$\vdots$};
};
\draw  [very thick](3.5,0.05+2*1.5) circle [radius=0.5][blue];
\node [scale=2][above] at (3.5,0.05+2*1.5-0.5) {$2$};
\draw  [very thick](3.5,0.05+1*1.5) circle [radius=0.5][];
\node [scale=2][above] at (3.5,0.05+1*1.5-0.5) {$1$};
%column 3
\foreach \x in {4,...,4}
{   \draw  [very thick](5,0.05+\x*1.5) circle [radius=0.5][green];
	\node [scale=2][above] at (5,0.1+\x*1.5-0.5) {$4$};
	\node [scale=2][above] at (5,6.8) {$\vdots$};
};
\foreach \x in {1,...,2}
{   \draw  [very thick](5,0.05+\x*1.5) circle [radius=0.5][];
	\node [scale=2][above] at (5,0.1+\x*1.5-0.5) {$1$};
};
\draw  [very thick](5,0.05+3*1.5) circle [radius=0.5][red];
\node [scale=2][above] at (5,0.05+3*1.5-0.5) {$3$};
%column 4
\foreach \x in {1,...,4}
{   \draw  [very thick](6.5,0.05+\x*1.5) circle [radius=0.5][green];
	\node [scale=2][above] at (6.5,0.1+\x*1.5-0.5) {$4$};
	\node [scale=2][above] at (6.5,6.8) {$\vdots$};
};
		\end{tikzpicture}
			\caption{\small Case 2. Although the particles below particles $2$ and $3$ are of different class in the speed process, treating them as first class particles does not affect the dynamics between particles $2$ and $3$.}
	\end{subfigure}

	\begin{subfigure}[t]{{.8\linewidth}}
		\centering
	\begin{tikzpicture}[scale=0.4, every node/.style={transform shape}]
%	%column 0
%	\foreach \x in {1,...,4}
%	{   \draw  [very thick](0.5,0.05+\x*1.5) circle [radius=0.5][];
%		\node [scale=2][above] at (0.5,0.1+\x*1.5-0.5) {$1$};
%		\draw [dashed][line width=0.01cm] (0.5,7) -- (0.5,8);
%	};
	\foreach \x in {-1,...,4}
{
	\node [scale=2][above] at (3.5+\x*1.5,0) {$\x$};
};
	%column 1
	\foreach \x in {1,...,4}
	{   \draw  [very thick](2,0.05+\x*1.5) circle [radius=0.5][];
		\node [scale=2][above] at (2,0.1+\x*1.5-0.5) {$1$};
	};
	\node [scale=2][above] at (2,6.8) {$\vdots$};
	%column 2
	\foreach \x in {3,...,4}
	{   \draw  [very thick](3.5,0.05+\x*1.5) circle [radius=0.5][green];
		\node [scale=2][above] at (3.5,0.1+\x*1.5-0.5) {$4$};
		\node [scale=2][above] at (3.5,6.8) {$\vdots$};
	};
	\draw  [very thick](3.5,0.05+2*1.5) circle [radius=0.5][blue];
	\node [scale=2][above] at (3.5,0.05+2*1.5-0.5) {$2$};
	\draw  [very thick](3.5,0.05+1*1.5) circle [radius=0.5][];
	\node [scale=2][above] at (3.5,0.05+1*1.5-0.5) {$1$};
	%column 3
	\foreach \x in {2,...,4}
	{   \draw  [very thick](5,0.05+\x*1.5) circle [radius=0.5][green];
	\node [scale=2][above] at (5,0.1+\x*1.5-0.5) {$4$};
	\node [scale=2][above] at (5,6.8) {$\vdots$};
	};
	\draw  [very thick](5,0.05+1*1.5) circle [radius=0.5][];
	\node [scale=2][above] at (5,0.05+1*1.5-0.5) {$1$};	
	%column 4
\foreach \x in {3,...,4}
{   \draw  [very thick](6.5,0.05+\x*1.5) circle [radius=0.5][green];
	\node [scale=2][above] at (6.5,0.1+\x*1.5-0.5) {$4$};
	\node [scale=2][above] at (6.5,6.8) {$\vdots$};
};
\foreach \x in {1,...,2}
{   \draw  [very thick](6.5,0.05+\x*1.5) circle [radius=0.5][];
	\node [scale=2][above] at (6.5,0.1+\x*1.5-0.5) {$1$};
	\node [scale=2][above] at (6.5,6.8) {$\vdots$};
};
	%column 5
	\foreach \x in {4,...,4}
	{   \draw  [very thick](8,0.05+\x*1.5) circle [radius=0.5][green];
		\node [scale=2][above] at (8,0.1+\x*1.5-0.5) {$4$};
		\node [scale=2][above] at (8,6.8) {$\vdots$};;
	};
	\foreach \x in {1,...,2}
	{   \draw  [very thick](8,0.05+\x*1.5) circle [radius=0.5][];
		\node [scale=2][above] at (8,0.1+\x*1.5-0.5) {$1$};
	};
	\draw  [very thick](8,0.05+3*1.5) circle [radius=0.5][red];
	\node [scale=2][above] at (8,0.05+3*1.5-0.5) {$3$};
	%column 6
	\foreach \x in {1,...,4}
	{   \draw  [very thick](9.5,0.05+\x*1.5) circle [radius=0.5][green];
		\node [scale=2][above] at (9.5,0.1+\x*1.5-0.5) {$4$};
		\node [scale=2][above] at (9.5,6.8) {$\vdots$};
	};
	\end{tikzpicture}
	\caption{\small Case 3 (2). Any particle in columns $0-3$ whose speed is strictly greater than that of the red and blue particles are treated as first class particles while all particles with speed below are treated as holes.}
\end{subfigure}
	\caption{\small Illustration of the three configurations described in cases 1, 2 and 3 of the proof of Theorem \ref{thm:ov}, with the minus signs omitted for neater presentation. The blue particle will ultimately meet the red particle.}
	\label{fig:points}
\end{figure}
\newpage

\bibliographystyle{plain}
\bibliography{TAZRP}

\begin{thebibliography}{10}

\bibitem{amir2011tasep}
Gideon Amir, Omer Angel, and Benedek Valk{\'o}.
\newblock The {TASEP} speed process.
\newblock {\em Ann.\ Probab.}, 39(4):1205--1242, 2011.

\bibitem{BN2017}
M\'{a}rton Bal\'{a}zs and Attila~L\'{a}szl\'{o} Nagy.
\newblock How to initialize a second class particle?
\newblock {\em Ann. Probab.}, 45(6A):3535--3570, 2017.

\bibitem{borodin2019color}
Alexei Borodin and Alexey Bufetov.
\newblock Color-position symmetry in interacting particle systems.
\newblock {\em Preprint arXiv:1905.04692}, 2019.

\bibitem{CP}
Eric Cator and Leandro P.~R. Pimentel.
\newblock Busemann functions and the speed of a second class particle in the
  rarefaction fan.
\newblock {\em Ann. Probab.}, 41(4):2401--2425, 2013.

\bibitem{C2011}
David Coupier.
\newblock Multiple geodesics with the same direction.
\newblock {\em Electron. Commun. Probab.}, 16:517--527, 2011.

\bibitem{CH2012}
David Coupier and Philippe Heinrich.
\newblock Coexistence probability in the last passage percolation model is
  {$6-8\log2$}.
\newblock {\em Ann. Inst. Henri Poincar\'{e} Probab. Stat.}, 48(4):973--988,
  2012.

\bibitem{FanSeppalainen}
Wai-Tong~(Louis) Fan and Timo Sepp\"{a}l\"{a}inen.
\newblock Joint distribution of {B}usemann functions in the exactly solvable
  corner growth model.
\newblock {\em Preprint arXiv:1808.09069}, 2018.

\bibitem{FGM2009}
Pablo~A. Ferrari, Patricia Gon\c{c}alves, and James~B. Martin.
\newblock Collision probabilities in the rarefaction fan of asymmetric
  exclusion processes.
\newblock {\em Ann. Inst. Henri Poincar\'{e} Probab. Stat.}, 45(4):1048--1064,
  2009.

\bibitem{FK}
Pablo~A. Ferrari and Claude Kipnis.
\newblock Second class particles in the rarefaction fan.
\newblock {\em Ann. Inst. H. Poincar\'{e} Probab. Statist.}, 31(1):143--154,
  1995.

\bibitem{ferrari2009multiclass}
Pablo~A. Ferrari and James~B. Martin.
\newblock Multiclass {H}ammersley--{A}ldous--{D}iaconis process and
  multiclass-customer queues.
\newblock {\em Ann.\ Inst.\ H.\ Poincar{\'e} Probab.\ Statist.},
  45(1):250--265, 2009.

\bibitem{F.M.P.}
Pablo~A. Ferrari, James~B. Martin, and Leandro P.~R. Pimentel.
\newblock A phase transition for competition interfaces.
\newblock {\em Ann. Appl. Probab.}, 19(1):281--317, 2009.

\bibitem{FerPim}
Pablo~A. Ferrari and Leandro P.~R. Pimentel.
\newblock Competition interfaces and second class particles.
\newblock {\em Ann. Probab.}, 33(4):1235--1254, 2005.

\bibitem{G2014}
Patr\'{\i}cia Gon\c{c}alves.
\newblock On the asymmetric zero-range in the rarefaction fan.
\newblock {\em J. Stat. Phys.}, 154(4):1074--1095, 2014.

\bibitem{K.}
Claude Kipnis.
\newblock Central limit theorems for infinite series of queues and applications
  to simple exclusion.
\newblock {\em Ann. Probab.}, 14(2):397--408, 1986.

\bibitem{liggett2012interacting}
Thomas~M. Liggett.
\newblock {\em Interacting particle systems}.
\newblock Springer, New York, 1985.

\bibitem{PerLy}
Russell Lyons and Yuval Peres.
\newblock {\em Probability on trees and networks}, volume~42 of {\em Cambridge
  Series in Statistical and Probabilistic Mathematics}.
\newblock Cambridge University Press, New York, 2016.

\bibitem{martin2010fixed}
James~B Martin and Balaji Prabhakar.
\newblock Fixed points for multi-class queues.
\newblock {\em Preprint arXiv:1003.3024}, 2010.

\bibitem{MG}
Thomas Mountford and Herv\'{e} Guiol.
\newblock The motion of a second class particle for the {TASEP} starting from a
  decreasing shock profile.
\newblock {\em Ann. Appl. Probab.}, 15(2):1227--1259, 2005.

\bibitem{MP95}
Thomas Mountford and Balaji Prabhakar.
\newblock On the weak convergence of departures from an infinite series of
  {$\cdot/M/1$} queues.
\newblock {\em Ann. Appl. Probab.}, 5(1):121--127, 1995.

\end{thebibliography}

\end{document}